\documentclass[12pt,twoside,a4paper]{article}

\usepackage{amsmath,amscd,amssymb,amsthm,array,bbold,colortbl,enumitem,eucal}
\usepackage{epsfig,eucal,graphics,multicol,verbatim}
\usepackage{url}
\usepackage{color,xcolor}
\usepackage{eurosym}
\usepackage{appendix}
\usepackage{wrapfig}
\usepackage{makeidx}
\usepackage{geometry}
\usepackage{multirow}
\usepackage{rotating}
\usepackage{slashbox}
\usepackage{tikz}
\usepackage{stmaryrd}
\usepackage{dsfont}
\usepackage{calrsfs}
\usepackage{mathabx}

\usetikzlibrary{matrix,arrows}

\theoremstyle{plain}
\newtheorem{theorem}{Theorem}[section]
\newtheorem{lemma}[theorem]{Lemma}
\newtheorem{proposition}[theorem]{Proposition}
\newtheorem{corollary}[theorem]{Corollary}

\theoremstyle{definition}

\newtheorem{definition}[theorem]{Definition}

\DeclareMathAlphabet{\mathbbmsl}{U}{bbm}{m}{sl}

\begin{document}

\newcommand{\Lip}{\text{\rm Lip}}
\newcommand{\Grass}{\text{\rm Grass}}
\newcommand{\Graph}{\text{\rm Graph}}
\newcommand{\diam}{\text{\rm diam}}
\newcommand{\dist}{\text{\rm dist}}
\newcommand{\Id}{\text{\rm Id}}
\newcommand{\Card}{\text{\rm Card}}

\numberwithin{equation}{section}

\title{Lipschitz sub-actions for locally maximal hyperbolic sets of a  $C^1$ maps}

\author{
Xifeng Su \footnote{School of Mathematical Sciences,
Beijing Normal University,
No. 19, XinJieKouWai Street, HaiDian District,
 Beijing 100875, P. R. China,
\texttt{xfsu@bnu.edu.cn, billy3492@gmail.com}}
\quad 
Philippe Thieullen \footnote{Institut de Math\'ematiques de Bordeaux,
Universit\'e de Bordeaux,
351 cours de la Lib\'eration, F 33405 Talence, France,
\texttt{philippe.thieullen@u-bordeaux.fr}}
\quad
Wenzhe Yu \footnote{School of Mathematical Sciences,
Beijing Normal University,
No. 19, XinJieKouWai Street, HaiDian District,
 Beijing 100875, P. R. China,
 \texttt{201821130020@mail.bnu.edu.cn}}
} 

\date{\today}

\maketitle

\begin{abstract}
Liv\v{s}ic theorem asserts that, for Anosov diffeomorphisms/flows, a Lipschitz observable is a coboundary if all its Birkhoff sums on every periodic orbits are equal to zero. The transfer function is then Lipschitz. We prove a positive Liv\v{s}ic theorem which asserts that a Lipschitz observable is bounded from below by a  coboundary if and only if all its Birkhoff sums on periodic orbits are non negative. The new result is that the coboundary can be chosen Lipschitz. The map  is only assumed to be $C^1$ and hyperbolic,  but not necessarily bijective nor transitive. We actually prove our main result in the setting of locally maximal hyperbolic sets for not general $C^1$ map. The construction of the coboundary uses a new notion of the Lax-Oleinik operator that is a standard tool in the discrete Aubry-Mather theory.

\end{abstract}

\section{Introduction and main results}

A $C^r$ dynamical system, $r\geq$,  is a couple $(M,f)$ where $M$ is a $C^r$ manifold of dimension $d_M \geq 2$, without boundary, not necessarily compact,  and $f:M \to M$ on a  $C^r$ map, not necessarily injective. The tangent bundle $TM$ is assumed to be equipped with a Finsler norm $\| \cdot \|$ depending $C^{r-1}$ with respect to the base point. A topological dynamical system is a couple $(M,f)$ where $M$ is a metric space and $f:M \to M$ is a continuous map.   We recall  several standard definitions. The theory of Anosov systems is well explained in Hasselblatt, Katok \cite{HasselblattKatok1995}, Bonatti, Diaz, Viana \cite{BonattiDiazViana2005}.

\begin{definition} \label{Definition:LocallyHyperbolicMap}
Let $(M,f)$ be a $C^r$ dynamical system and  $\Lambda \subseteq M$ be a compact set strongly invariant by $f$, $f(\Lambda)=\Lambda$.  Let  $d_M = d^u + d^s$, $d^u \geq1$, $d^s \geq 1$.
\begin{enumerate}

\item $\Lambda$ is said to be hyperbolic  if there exist constants $\lambda^s < 0 < \lambda^u$, $C_{\Lambda}\geq1$, and a continuous equivariant splitting over $\Lambda$, $\forall\, x  \in \Lambda, \ T_xM =E_\Lambda^u(x) \oplus E_\Lambda^s(x)$, 
\begin{gather*}
\left\{\begin{array}{ccc}
\Lambda &\to& \Grass(TM,d^u) \\
x &\mapsto& E_\Lambda^u(x)
\end{array}\right.
\quad
\left\{\begin{array}{ccc}
\Lambda &\to& \Grass(TM,d^s) \\
x &\mapsto& E_\Lambda^s(x)
\end{array}\right. \quad \text{are \ $C^0$},
\end{gather*}
such that
\begin{gather*}
\forall\, x \in \Lambda, \ T_xf(E^u(x)) = E^u(f(x)), \ T_xf(E^s(x)) \subseteq E^s(f(x)), \\
\forall x \in \Lambda, \ \forall n\geq0, \ 
\renewcommand{\arraystretch}{1.2}
\left\{\begin{array}{l}
\forall v \in E_\Lambda^s(x) ,\ \| T_xf^n(v) \| \leq C_{\Lambda} \, e^{n\lambda^s} \|v\|, \\
\forall v \in E_\Lambda^u(x), \ \| T_xf^{n}(v) \| \geq C_{\Lambda}^{-1} \, e^{n\lambda^u} \|v\|.
\end{array}\right.
\end{gather*}

\item $\Lambda$ is said to  be  {\it locally maximal} if  there exists an open neighborhood $U$ of $\Lambda$ of compact closure  such that  
\[
\bigcap_{n\in\mathbb{Z}} f^n(\bar U) = \Lambda.
\]
\item $\Lambda$ is said to be an {\it   attractor}  if   there exists an open neighborhood $U$  of $\Lambda$ of compact closure such that 
\[
f(\bar U) \subseteq U \quad\text{and}\quad \bigcap_{n\geq0} f^n(\bar U) = \Lambda.
\]
\end{enumerate}
(Notice that the map $f$ is not assumed to be invertible nor transitive as it is done usually.)
\end{definition}

We also consider a Lipschitz continuous observable $\phi : U \to \mathbb{R}$. We want to understand the structure of the orbits that minimize the Birkhoff averages of $\phi$. We recall several standard definitions.

\begin{definition} \label{Definition:DiscreteErgodicOptimization}
Let $(M,f)$ be a topological dynamical system, $\Lambda \subseteq M$ be an $f$-invariant compact set, $U\supseteq \Lambda$ be an open neighborhood of $\Lambda$, and  $\phi : U \to \mathbb{R}$ be a continuous function. 
\begin{enumerate}
\item \label{item01:BasicErgodicOptimization} The {\it ergodic minimizing value} of $\phi$ restricted to $\Lambda$ is the quantity
\begin{equation} \label{Equation:ErgodicMinimizingValue}
\bar \phi_\Lambda := \lim_{n\to+\infty} \frac{1}{n} \inf_{x \in\Lambda} \sum_{k=0}^{n-1} \phi \circ f(x).
\end{equation}
\item \label{item02:BasicErgodicOptimization} A continuous  function $u : U \to \mathbb{R}$ is said to be a {\it subaction} if
\begin{equation} \label{Equation:Subaction}
\forall\, x \in U, \ \phi(x) - \bar \phi_\Lambda \geq u \circ f(x) - u(x).
\end{equation}
\item \label{item03:BasicErgodicOptimization} A function $\psi$ of the form $\psi = u \circ f - u$ for some $u$ is called a coboundary.
\item The Lipschitz constant of  $\phi$ is the number
\[
\Lip(\phi) := \sup_{x,y \,\in U, \ x\not= y} \frac{|\phi(y)-\phi(x)|}{d(x,y)},
\]
where $d(\cdot,\cdot)$ is the distance  associated to the Finsler norm.
\end{enumerate}
\end{definition}

Our main result is the following.

\begin{theorem} \label{Theorem:DiscreteSubactionExistence}
Let $(M,f)$ be a $C^1$ dynamical system, $\Lambda \subseteq M$ be a locally maximal hyperbolic compact set, $\phi : M \to \mathbb{R}$ be a Lipschitz continuous function, and $\bar \phi_\Lambda$ be the ergodic minimizing value of $\phi$ restricted to $\Lambda$. Then there exists an open set $\Omega$ containing $\Lambda$ and a Lipschitz continuous function $u:\Omega \to \mathbb{R}$  such that
\[
\forall\, x \in\Omega, \quad \phi(x) - \bar \phi_\Lambda \geq u \circ f(x) - u(x).
\]
Moreover, $\Lip(u) \leq K_\Lambda \Lip(\phi)$ for some constant $K_\Lambda$ depending only on the hyperbolicity of $f$ on $\Lambda$. The constant $K_\Lambda$ is semi-explicit
\[
K_\Lambda=\max \Big\{\frac{(N_{AS}+1)\diam(\Omega_{AS})}{\varepsilon_{AS}},\ K_{APS} \Big\}
\]
where  $\varepsilon_{AS}, K_{APS}$ and $N_{AS}$ are defined in  \ref{Theorem:AnosovShadowingImproved},  \ref{Proposition:AnosovPeriodicShadowingLemma}, and   \ref{lemma:FinitenessPseudoBirkhoffSums}.
\end{theorem}

\begin{corollary} \label{corollary:AnosovPositiveLivsic}
Let $(M,f)$ be a $C^1$ dynamical system, $\Lambda \subseteq M$ be a locally maximal hyperbolic compact set, and $\phi : M \to \mathbb{R}$ be a Lipschitz continuous function. Assume the Birkhoff sum of $\phi$ on every periodic orbits on $\Lambda$ is non negative. Then there exist an open neighborhood $\Omega$ of $\Lambda$, a Lipschitz continuous function $u:\Omega \to \mathbb{R}$, such that 
\[
\forall\, x \in\Omega, \ \phi(x) - u \circ f(x) + u(x) \geq 0.
\] 
\end{corollary}

A weaker version of Theorem \ref{Theorem:DiscreteSubactionExistence} was obtained in \cite{LopesThieullen2003}, \cite{PollicottSharp2004}, and \cite{LopesRosasRuggiero2007}, where the subaction is only H\"older. Bousch   claims in \cite{Bousch2011}  that the subaction can be chosen Lipschitz continuous as a corollary of its original approach, but the proof does not appear to us very obvious.  Huang, Lian, Ma, Xu, and Zhang  proved in \cite[Appendix A]{HuangLianMaXuZhang2019_1}  a weaker version, namely $\frac{1}{N}\sum_{k=0}^{N-1} [\phi-\bar\phi ] \geq u_N \circ f^N - u_N$ for some integer $N\geq1$ and some $u_N$ Lipschitz but by invoking again \cite{Bousch2011}. A similar theorem can be proved for Anosov flows, see \cite{SuThieullen2021}.

The plan of the proof is the following. We revisit the {\it Anosov shadowing lemma} in section \ref{section:AnosovShadowingLemmaRevisited}, Theorem \ref{Theorem:AnosovShadowingImproved},  by bounding from the above the sum of the distances between a pseudo orbit and a true shadowed orbit in terms of the sum of the pseudo errors. We  improve in section \ref{section:DiscreteLaxOleinikOperator} Bousch's techniques of the construction of a coboundary by introducing a new {\it Lax-Oleinik operator}, Definition  \ref{definition:LaxOleinikOperator}, and showing under the assumption of {\it positive Liv\v{s}ic criteria} the existence of a stronger notion of calibrated subactions, Proposition \ref{proposition:WeakKAMsolution}. We then check in section \ref{section:ProofDiscreteCase} that a locally maximal hyperbolic set satisfies the positive Liv\v{s}ic criteria and prove the main result. The proof of Theorem \ref{Theorem:AnosovShadowingImproved} requires a precise description of the notions of {\it adapted local hyperbolic maps and graph transforms with respect to a family of adapted charts}. We revisit these notions in  Appendix \ref{Appendix:LocalHyperbolicDynamics}.  Notice that we do not assume $f$ to be invertible nor transitive. 

\section{An improved shadowing lemma for maps}
\label{section:AnosovShadowingLemmaRevisited}

We show in this section an improved version of the shadowing lemma that will be needed to check the existence of a fixed point of the Lax-Oleinik operator. 

\begin{definition}
Let $(M,f)$ be a topological dynamical system. A sequence $(x_i)_{0 \leq  i \leq n}$ of points of $M$ is said to be an $\epsilon$-pseudo orbit (with respect to the dynamics $f$) if 
\[
\forall\, i \in \llbracket 0, n-1 \rrbracket, \quad d(f(x_i),x_{i+1})\leq \epsilon.
\]
The sequence is said to be a periodic $\epsilon$-pseudo orbit if $x_n=x_0$.
\end{definition}

We first recall the basic Anosov shadowing property.

\begin{lemma}[Anosov shadowing lemma] \label{Theorem:AnosovShadowingLemma}
Let $(M,f)$ be a $C^1$ dynamical system and $\Lambda \subseteq M$ be a compact hyperbolic set. Then there exist constants $\epsilon_{AS}>0$, $K_{AS}\geq1$, and $\lambda_{AS}>0$, such that for every $n\geq1$, for every $\epsilon_{AS}$-pseudo orbit $(x_i)_{0\leq i\leq n}$ of the neighborhood  $\Omega_{AS}= \{ x \in M : d(x,\Lambda) < \epsilon_{AS} \}$, there exists a point $y\in M$ such that
\begin{gather}
\max_{0 \leq i \leq n} d(x_i,f^i(y)) \leq K_{AS}\max_{1 \leq k \leq n} d(f(x_{k-1}) ,x_{k}). \label{item:AnosovShadowingLemma_3}
\end{gather}
\end{lemma}

Equation \eqref{item:AnosovShadowingLemma_3} is the standard conclusion of the shadowing lemma. We say that the orbit $(y,f(y), \ldots, f^{n}(y))$ shadows the pseudo orbit $(x_i)_{i=0}^n$.

\begin{theorem}[Improved Anosov shadowing lemma] \label{Theorem:AnosovShadowingImproved}
Let $(M,f,\Lambda)$ as in Lemma \ref{Theorem:AnosovShadowingLemma}. Then one can choose $\epsilon_{AS}>0$, $K_{AS}\geq1,\lambda_{AS}>0$, and $y\in M$ so that 
\begin{gather}
\forall\, i \in \llbracket 0, n \rrbracket, \ d(x_i,f^i(y)) \leq K_{AS} \sum_{k=1}^{n} d(f(x_{k-1}),x_{k}) \, \exp(-\lambda_{AS} |k-i|), \label{item:AnosovShadowingLemma_1} \\
\sum_{i=0}^{n}d(x_i,f^i(y)) \leq  K_{AS} \sum_{k=1}^{n} d(f(x_{k-1}),x_{k})  \label{item:AnosovShadowingLemma_2}.
\end{gather}
\end{theorem}

Equations \eqref{item:AnosovShadowingLemma_1} and \eqref{item:AnosovShadowingLemma_2} are new and fundamental for improving Bousch's approach \cite{Bousch2011}. The heart of the proof is done through the notion of adapted local charts. In appendix \ref{Appendix:LocalHyperbolicDynamics} we recall the notion of {\it adapted local dynamics} in which the dynamics is observed through the iteration of a sequence of maps which are uniformly hyperbolic with respect to a family of  norms that are adapted to the unstable/stable splitting and the constants of hyperbolicity.

The following Theorem \ref{Theorem:AdaptedAnosovShadowingLemma} is the technical counterpart of Theorem~\ref{Theorem:AnosovShadowingImproved}.
We consider a sequence of uniformly hyperbolic maps as described more rigorously in Appendix \ref{Appendix:LocalHyperbolicDynamics} 
\[
f_i : B_i(\rho) \to \mathbb{R}^d , \ B_i(\rho) \subset \mathbb{R}^d = E^u_i \oplus E_i^s =  E^u_{i+1} \oplus E_{i+1}^s, \ A_i = T_0 f_i,
\]
where $E_i^{u/s}$ are the unstable/stable vector spaces, $A_i$ is the tangent map of $f_i$ at the origin which is assumed to be uniformly hyperbolic with respect to an adapted norm $\| \cdot \|_i$ and constants of hyperbolicity $\sigma^s < 1 < \sigma^u$, $\eta>0$ is the size of the perturbation of the non linear term $\big(f_i(v) -f_i(0)-A_i\cdot v\big)$, $\rho>0$ is the size of the domain of definition of $f_i$, $B_i(\rho)$ is the ball of radius $\rho$ for the norm $\| \cdot \|_i$, and $\|f_i(0)\|_i \leq \epsilon(\rho)$ is the size of the shadowing constant with $\epsilon(\rho) \ll \rho$.

\begin{theorem}[Adapted Anosov shadowing lemma] \label{Theorem:AdaptedAnosovShadowingLemma}
Let $(f_i,A_i,E_i^{u/s},\| \cdot \|_i)_{i=0}^{n-1}$ be a family of adapted local hyperbolic maps  and $(\sigma^u,\sigma^s,\eta,\rho)$ be a set of hyperbolic constants as in Definition  \ref{Definition:AdaptedLocalHyperbolicMap}. Assume the stronger estimate
\[
\eta < \min \Big( \frac{(1-\sigma^s)^2}{12}, \frac{\sigma^u-1}{6} \Big).
\]
Define $\lambda_\Gamma$ and $K_\Gamma$  by,
\[
\exp(-\lambda_\Gamma) = \max \Big(\frac{\sigma^s+3\eta}{1-3\eta}, \frac{1}{\sigma^u-3\eta} \Big), \quad K_\Gamma = \frac{5}{\big(1-\exp(-\lambda_\Gamma)\big)^2}.
\]
Let $(q_i)_{i=0}^n$ be a ``pseudo sequence'' of points in the sense 
\[
\forall\, i \in \llbracket 0,n-1 \rrbracket,  \quad q_i \in B_{i}\Big(\frac{\rho}{2}\Big) \ \ \text{and} \ \  f_{i}(q_i) \in B_{{i+1}}\Big(\frac{\rho}{2}\Big).
\]
Then there exists a ``true sequence'' of points $(p_i)_{i=0}^n$, $p_i \in B_i(\rho)$, such that
\begin{enumerate}
\item \label{Item:AdaptedAnosovShadowingLemma_01} $\displaystyle \forall\, i \in \llbracket 0, n-1 \rrbracket, \ f_{i}(p_i) = p_{i+1}$, (the true orbit),
\item \label{Item:AdaptedAnosovShadowingLemma_02} $\displaystyle \forall\, i \in \llbracket 0, n \rrbracket, \ \|q_i-p_i \|_i \leq K_\Gamma \sum_{k=1}^{n} \|f_{k-1}(q_{k-1}) - q_{k} \|_{k} \exp(-\lambda_\Gamma|k-i|)$,
\item \label{Item:AdaptedAnosovShadowingLemma_03} $\displaystyle \sum_{i=0}^n \|q_i-p_i\|_i \leq  K_\Gamma \sum_{k=1}^{n}  \|f_{k-1}(q_{k-1}) - q_{k} \|_{k}$,
\item \label{Item:AdaptedAnosovShadowingLemma_04} $\displaystyle \max_{0 \leq i \leq n}  \|q_i-p_i\|_i \leq K_\Gamma \max_{1 \leq k \leq n}  \|f_{k-1}(q_{k-1}) - q_{k} \|_{k}$.
\end{enumerate}
Moreover assume $(f_i,A_i,E_i^{u/s},\| \cdot \|_i)_{i\in\mathbb{Z}}$ is $n$-periodic in the sense
\[
f_{i+n}=f_i, \ A_{i+n} = A_i, \ E_{i+n}^{u/s} = E_i^{u/s}, \ \| \cdot \|_{i+n} = \| \cdot \|_i,
\]
assume in addition that  $(q_i)_{i\in\mathbb{Z}}$ is a periodic pseudo sequence in the following sense
\[
\forall\,  i\in\mathbb{Z}, \ q_{i+n}=q_i, \ q_i \in B_i\Big(\frac{\rho}{2} \Big), \ f_{i-1}(q_{i-1}) \in B_i\Big(\frac{\rho}{2} \Big).
\]
Then there exists a periodic true sequence $(p_i)_{i\in\mathbb{Z}}$ satisfying 
\begin{enumerate}
\addtocounter{enumi}{4}
\item \label{Item:AdaptedAnosovShadowingLemma_05} $\displaystyle \forall\, i \in \mathbb{Z}, \ f_{i}(p_i) = p_{i+1}$, $p_{i+n}=p_i$,
\item \label{Item:AdaptedAnosovShadowingLemma_06} $\displaystyle \sum_{i=0}^{n-1} \|q_i-p_i\|_i \leq \tilde K_\Gamma \sum_{k=1}^{n}  \|f_{k-1}(q_{k-1}) - q_{k} \|_{k}$,
\end{enumerate}
with $\tilde K_\Gamma := K_\Gamma(1+\exp(-\lambda_\Gamma))/(1-\exp(-\lambda_\Gamma))$.
\end{theorem}

\begin{figure}[hbt]
\centering
\includegraphics[width=0.8\textwidth]{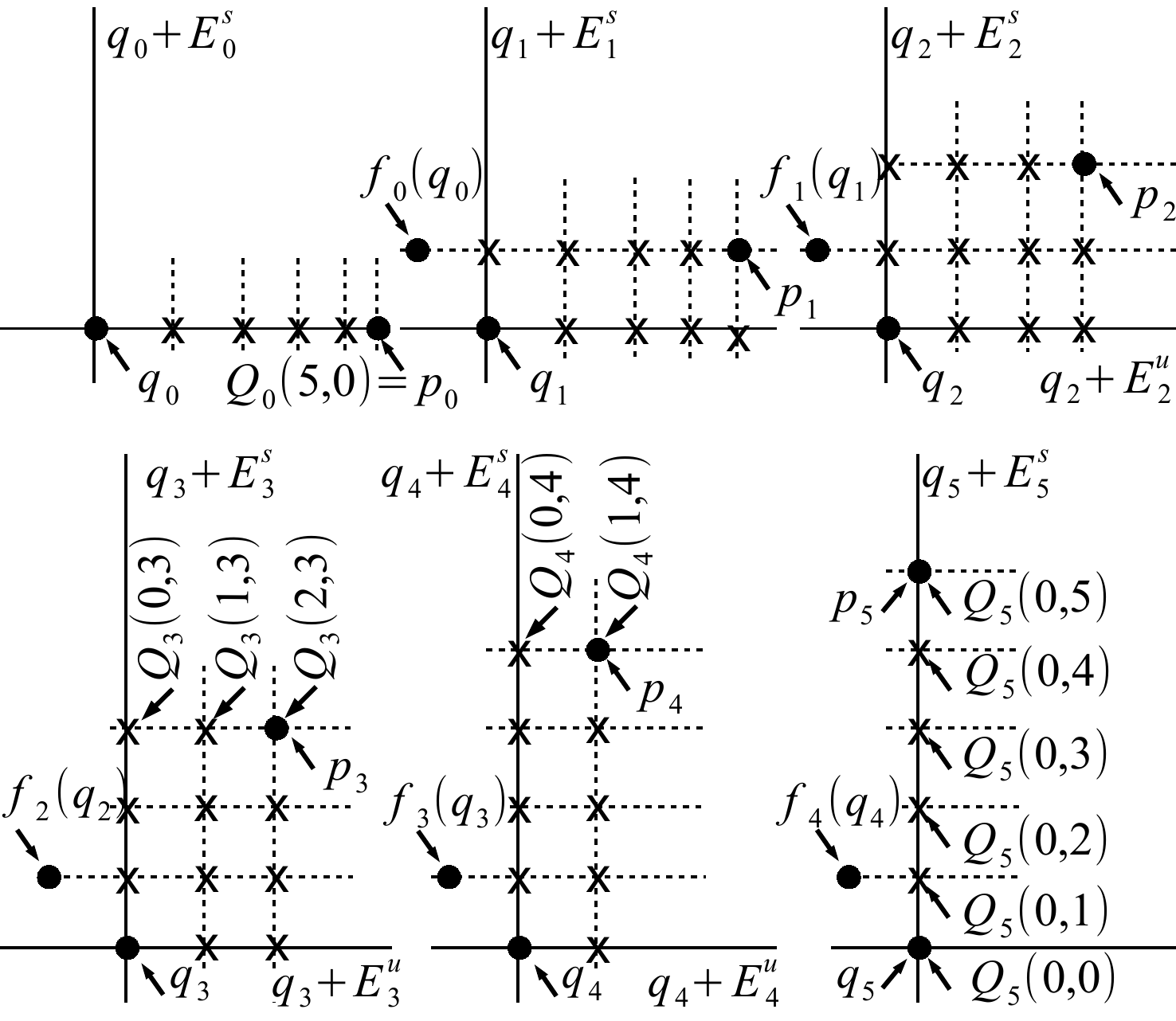}
\caption{A schematic description of the grid $Q_k(i,j)$ for $n=5$.}
\label{Figure:SchematicGrid}
\end{figure}

\begin{proof}
Let $P_i^u,P_i^s$ be the projections onto $E_i^u,E_i^s$ respectively. Let 
\[
\alpha = \frac{6\eta}{\sigma^u-\sigma^s}, \quad \delta_{i} = \| f_{i-1}(q_{i-1}) - q_{i} \|_{i}.
\]
Notice that the proof of items \ref{Item:AdaptedAnosovShadowingLemma_03} and \ref{Item:AdaptedAnosovShadowingLemma_04} follows readily from item \ref{Item:AdaptedAnosovShadowingLemma_02}. We prove only item \ref{Item:AdaptedAnosovShadowingLemma_02}.

\medskip
{\it Step 1.}  We construct by induction a grid of points 
\[
Q_i(j,k) \in B_i(\rho) \ \ \text{for} \ \ i \in \llbracket 0,n \rrbracket, \ \ j \in \llbracket 0,n-i \rrbracket, \ \ \text{and} \ \ k \in \llbracket 0, i \rrbracket
\]
in the following way (see Figure \ref{Figure:SchematicGrid}):
\begin{enumerate}
\item For all $i \in \llbracket 0,n \rrbracket$, let $G_{i,0} : B_i^u(\rho) \to B_i^s(\rho)$ be the horizontal graph passing through the point $q_i$,
\[
\forall\, v\in B_i^u(\rho), \ G_{i,0}(v) = P_i^sq_i.
\]
For all $i \in \llbracket 1,n \rrbracket$ and $k \in \llbracket 1, i \rrbracket$,  let $G_{i,k} : B_i^u(\rho) \to B_i^s(\rho)$ be the graph obtained by  the graph transform (see Proposition \ref{Proposition:ForwardLocalGraphTransform}), iterated $k$ times, of $G_{i-k,0}$,
\[
G_{i,k} = (\mathcal{T})^u_{i-1} \circ \cdots \circ (\mathcal{T})^u_{i-k}(G_{i-k,0}).
\]
Notice that $\| G_{i,k}(0)\|_i \leq \rho/2$ and $\Lip(G_{i,k}) \leq \alpha$.
\item For all  $i \in \llbracket 0, n \rrbracket$ and  $k \in \llbracket 0,i \rrbracket$, let $Q_i(0,k)$ be the point on $\Graph(G_{i,k})$ whose unstable projection is $P^u_i q_i$, or more precisely,
\[
Q_i(0,k) =P^u_i q_i + G_{i,k}(P^u_i q_i).
\]
\item Let $i \in \llbracket 1,n \rrbracket$ and assume that the points $Q_i(j,k)$ have been defined for all $j \in \llbracket 0, n-i \rrbracket$ and $k\in \llbracket 0,i \rrbracket$. Let $ j \in \llbracket 1, n-i+1 \rrbracket$ and $k \in \llbracket 0, i-1 \rrbracket$, then $Q_{i-1}(j,k)$ is the unique point on $\Graph(G_{i-1,k})$ such that
\[
f_{i-1}(Q_{i-1}(j,k)) = Q_i( j-1,k+1).
\]
For $j=0$, the  points $Q_{i-1}(0,k)$ have been defined in item ii.
\end{enumerate}
We will then choose $p_i = Q_i(n-i,i)$.

\medskip
{\it Step 2.} Let  $h_{i,j} := \| P^s_i\big[Q_i(j,0) -Q_i(j,i) \big]\, \|_i$. We show that, for all $i\in\llbracket 1,n \rrbracket$,
\begin{gather*}
h_{i,0} \leq \Big[ (1+\alpha) + \frac{\alpha}{1-\alpha^2} \frac{\sigma^s + 3\eta}{\sigma^u-3\eta} \Big] \delta_i + \frac{\sigma^s+3\eta}{1-\alpha^2} \, h_{i-1,0}. 
\end{gather*}
Proposition \ref{Proposition:ForwardLocalGraphTransform} with  slope $\alpha = 6\eta /(\sigma^u-\sigma^s)$ for the unstable graphs show that
\begin{align*}
\| P_i^s \big[  Q_i(0,0) - Q_i(0,1) \big]\, \|_i &\leq \| P_i^s \big[  q_i - f_{i-1}(q_{i-1}) \big]\, \|_i + \|  P_i^s \big[ f_{i-1}(q_{i-1}) - Q_i(0,1) \big]\, \|_i \\
&\leq \delta_i + \alpha \|  P_i^u \big[ f_{i-1}(q_{i-1}) - Q_i(0,1) \big]\, \|_i \\
&\leq \delta_i + \alpha \|  P_i^u \big[ f_{i-1}(q_{i-1}) - q_i \big]\, \|_i \leq (1+\alpha)\delta_i.
\end{align*}
By forward induction, using Lemma \ref{Lemma:EquivarianceUnstableCone},
\begin{gather*}
Q_{i-1}(j,k) - Q_{i-1}(j',k) \in  C^u_{i-1}(\alpha)  \ \Rightarrow \ Q_{i}(j-1,k+1) - Q_{i}(j'-1,k+1) \in C^u_i (\alpha), \\
\| Q_{i-1}(j,k) - Q_{i-1}(j',k)  \| \leq \frac{1}{\sigma^u-3\eta} \| Q_{i}(j-1,k+1) - Q_{i}(j'-1,k+1) \|.
\end{gather*}
Then
\begin{align*}
\|Q_{i-1}(0,0) - Q_{i-1}(1,0) \|_{i-1} &\leq \frac{1}{\sigma^u-3\eta} \| P_i^u \big[ f_{i-1}(q_{i-1}) - Q_i(0,1) \big]\, \|_i \\
&\leq \frac{1}{\sigma^u-3\eta} \| P_i^u \big[ f_{i-1}(q_{i-1}) - q_i \big]\, \|_i =  \frac{1}{\sigma^u-3\eta} \delta_i.
\end{align*}
By backward induction, using Lemma \ref{Lemma:EquivarianceUnstableCone},
\begin{gather*}
Q_i(j,k) - Q_i(j,k') \in  C^s_i(\alpha) \ \Rightarrow \ Q_{i-1}(j+1,k-1) - Q_{i-1}(j+1,k'-1) \in C^s_{i-1}(\alpha), \\
\| Q_i(j,k) - Q_i(j,k')  \| \leq  (\sigma^s+3\eta) \| Q_{i-1}(j+1,k-1) - Q_{i-1}(j+1,k'-1)\|.
\end{gather*}
Then,
\begin{align*}
h_{i,0} &= \| Q_i(0,0)-Q_i(0,i) \|_i = \| P_i^s \big[ Q_i(0,0)-Q_i(0,i) \big]\, \|_i \\
&\leq \| P^s_i \big[ Q_i(0,0) -Q_i(0,1) \big]\, \|_i + \| P^s_i \big[ Q_i(0,1) - Q_i(0,i) \big]\, \|_i \\
&\leq (1+\alpha) \delta_i+ (\sigma^s+3\eta)  h_{i-1,1}.
\end{align*}
We estimate $h_{i-1,1}$ in the following way,
\begin{align*}
h_{i-1,1} &\leq \| P^s_{i-1} \big[ Q_{i-1}(1,0) - Q_{i-1}(0,0) \big]\, \|_{i-1} \\
&\quad+  \| P^s_{i-1}\big[Q_{i-1}(0,0) -Q_{i-1}(0,i-1) \big]\, \|_{i-1} \\
&\quad\quad+ \| P^s_{i-1}\big[Q_{i-1}(0,i-1) -Q_{i-1}(1,i-1) \big]\, \|_{i-1} \\
&\leq h_{i-1,0} + \alpha \| P^u_{i-1}\big[Q_{i-1}(0,i-1) -Q_{i-1}(1,i-1) \big]\, \|_{i-1}.
\end{align*}
\begin{align*}
\| P^u_{i-1}\big[Q_{i-1}(0,i-1) &- Q_{i-1}(1,i-1) \big]\, \|_{i-1} \\
&\leq \| P^u_{i-1}\big[Q_{i-1}(0,i-1) -Q_{i-1}(0,0) \big]\, \|_{i-1} \\
&\quad+ \| P^u_{i-1}\big[Q_{i-1}(0,0) -Q_{i-1}(1,0) \big]\, \|_{i-1} \\
&\quad\quad+  \| P^u_{i-1}\big[Q_{i-1}(1,0) -Q_{i-1}(1,i-1) \big]\, \|_{i-1}.
\end{align*}
\begin{gather*}
\| P^u_{i-1}\big[Q_{i-1}(0,0) -Q_{i-1}(1,0) \big]\, \|_{i-1} \leq \frac{1}{\sigma^u-3\eta} \| P^u \big[ f_{i-1}(q_{i-1}) - Q_i(0,1) \big]\, \|_i \\
\| P^u_{i-1}\big[Q_{i-1}(1,0) -Q_{i-1}(1,i-1) \big]\, \|_{i-1} \leq  \alpha \| P^s_{i-1}\big[Q_{i-1}(1,0) -Q_{i-1}(1,i-1) \big]\, \|_{i-1} \\
\| P^u_{i-1}\big[Q_{i-1}(0,i-1) - Q_{i-1}(1,i-1) \big]\, \|_{i-1}  \leq \frac{1}{\sigma^u-3\eta} \delta_i + \alpha h_{i-1,1}.
\end{gather*}
Then
\begin{gather*}
h_{i-1,1} \leq \frac{1}{1-\alpha^2} h_{i-1,0} + \frac{\alpha}{(1-\alpha^2)(\sigma^u-3\eta)} \delta_{i},
\end{gather*}
and finally
\begin{gather*}
h_{i,0} \leq \Big[ (1+\alpha) + \frac{\alpha}{1-\alpha^2} \frac{\sigma^s + 3\eta}{\sigma^u-3\eta} \Big] \delta_i + \frac{\sigma^s+3\eta}{1-\alpha^2} \, h_{i-1,0}. 
\end{gather*}

{\it Step 3.} We show that, for every $i\in \llbracket 0,n-1 \rrbracket$,
\begin{gather*}
\|P_i^u \big[  Q_i(0,i) - Q_i(1,i) \big]\, \|_i \leq \frac{\delta_{i+1}}{(1-\alpha^2)(\sigma^u-3\eta)} + \frac{\alpha}{1-\alpha^2}h_{i,0}.
\end{gather*}
Indeed, using
\begin{gather*}
\| P_i^u \big[ Q_i(1,0) - Q_i(1,i) \big]\, \|_i \leq \alpha \| P_i^s \big[ Q_i(1,0) - Q_i(1,i) \big]\, \|_i , \\
\| P_i^s \big[ Q_i(0,i) - Q_i(1,i) \big]\, \|_i \leq \alpha  \| P_i^u \big[ Q_i(0,i) - Q_i(1,i) \big]\, \|_i ,
\end{gather*}
we obtain
\begin{align*}
\| P_i^u \big[ Q_i(0,i) &- Q_i(1,i) \big]\, \|_i \\
&\leq \| P_i^u \big[ Q_i(0,0) - Q_i(1,0) \big]\, \|_i + \alpha \| P_i^s \big[ Q_i(1,0) - Q_i(1,i) \big]\, \|_i \\
&\leq \frac{1}{\sigma^u-3\eta} \| P^u \big[ f_i(q_i) - q_{i+1} \big]\, \|_{i+1} \\
&\quad\quad + \alpha\Big( \| P_i^s \big[ Q_i(0,0) - Q_i(0,i) \big]\, \|_i + \alpha \| P_i^u \big[ Q_i(0,i) - Q_i(1,i) \big]\, \|_i \Big), \\
\| P_i^u \big[ Q_i(0,i) &- Q_i(1,i) \big]\, \|_i \leq \frac{\delta_{i+1}}{(1-\alpha^2)(\sigma^u-3\eta)} + \frac{\alpha}{1-\alpha^2}h_{i,0}.
\end{align*}

\medskip
{\it Step 4.} We simplify the previous inequalities
\begin{gather*}
\frac{\sigma^s+3\eta}{\sigma^u-3\eta} \leq 1, \ \alpha \leq \frac{1}{2}, \  (1+\alpha) + \frac{\alpha}{1-\alpha^2} \frac{\sigma^s + 3\eta}{\sigma^u-3\eta}  \leq \frac{13}{6}.
\end{gather*}
Then for every $i \in \llbracket 0,n-1 \rrbracket$,
\begin{align*}
\| P_i^u \big[ Q_i(0,i) &- Q_i(n-i,i) \big]\, \|_i \leq  \sum_{k=0}^{n-i-1} \| P_i^u \big[ Q_{i}(k,i) - Q_i(k+1,i) \big]\, \|_i \\
&\leq \sum_{k=0}^{n-i-1} \Big( \frac{1}{\sigma^u-3\eta}\Big)^k \| P_{i+k}^u \big[ Q_{i+k}(0,i+k) - Q_{i+k}(1,i+k) \big]\, \|_{i+k} \\
&\leq \sum_{k=0}^{n-i-1} \Big( \frac{1}{\sigma^u-3\eta}\Big)^k \Big( \frac{\delta_{i+k+1}}{(1-\alpha^2)(\sigma^u-3\eta)} + \frac{\alpha}{1-\alpha^2}h_{i+k,0} \Big).
\end{align*}
By using  $\| P_i^s \big[ Q_i(0,i) - Q_i(n-i,i) \big]\, \|_i \leq \alpha \| P_i^u \big[ Q_i(0,i) - Q_i(n-i,i) \big]\, \|_i$, we obtain for every $i \in \llbracket 0, n \rrbracket$,
\begin{align*}
\| Q_i(0,i) - Q_i(n-i,i)  \|_i &\leq \frac{1}{1-\alpha}\sum_{k=i+1}^{n} \Big( \frac{1}{\sigma^u-3\eta}\Big)^{k-i} \delta_{k} \\
&\quad\quad +  \frac{\alpha}{1-\alpha}\sum_{k=i}^{n-1} \Big( \frac{1}{\sigma^u-3\eta}\Big)^{k-i}  h_{k,0}, \\
h_{i,0} = \| Q_i(0,0) - Q_i(0,i) \|_i &\leq \frac{13}{6} \sum_{k=1}^{i} \Big( \frac{\sigma^s+3\eta}{1-\alpha^2} \Big)^{i-k} \delta_k.
\end{align*} 
Let  
\[
\sigma_\Gamma := \max \Big( \frac{\sigma^s+3\eta}{1-\alpha^2}, \frac{1}{\sigma^u-3\eta} \Big) \leq  \exp(-\lambda_\Gamma).
\]
Combining these two last estimates, we obtain
\begin{gather*}
\| Q_i(0,0) - Q_i(n-i,0) \|_i \leq \frac{13}{6} \sum_{k=1}^{n} \sigma_\Gamma^{|k-i|} \delta_k + \sum_{k=i}^{n-1} \sigma_\Gamma^{k-i} h_{k,0}, \\
\sum_{k=i}^{n} \sigma_\Gamma^{k-i} h_{k,0} \leq \frac{13}{6} \sum_{k=i}^n \sigma_\Gamma^{k-i} \sum_{l=1}^k \sigma_\Gamma^{k-l} \delta_l = \frac{13}{6} \sum_{l=1}^n \sigma_\Gamma^{|l-i|} \Big( \sum_{k \geq \max(i,l)} \frac{\sigma_\Gamma^{k-i} \sigma_\Gamma^{k-l}}{\sigma_\Gamma^{|l-i|}} \Big) \delta_l.
\end{gather*}
In both cases $k \geq i \geq l$ or $k \geq l \geq i$,
\[
\frac{\sigma_\Gamma^{k-i} \sigma_\Gamma^{k-l}}{\sigma_\Gamma^{|l-i|}} = \sigma_\Gamma^{2(k-i)} \quad \text{or} \quad 
\frac{\sigma_\Gamma^{k-i} \sigma_\Gamma^{k-l}}{\sigma_\Gamma^{|l-i|}} = \sigma_\Gamma^{2(k-l)}.
\]
We finally obtain for every  $i \in \llbracket 0,n \rrbracket$,
\[
\| p_i-q_i \|_i \leq \frac{13}{3} \frac{1}{1-\sigma_\Gamma^2} \sum_{k=1}^n \sigma_\Gamma^{|k-i|}\delta_k.
\]
We conclude by noticing 
\[
\sum_{i=0}^n \sum_{k=1}^n \sigma_\Gamma^{|k-i|} \leq \frac{1+\sigma_\Gamma}{1-\sigma_\Gamma}.
\]
Consider now a periodic sequence $(q_j)_{j\in\mathbb{Z}}$. For every integer $s\geq1$, consider the restriction of that sequence over $\llbracket -sn,sn\rrbracket$ and apply the first part with a shift in the indices $i = j+sn$. There exists a sequence $(p_j^{s})_{j=-sn}^{sn}$ such that, for every $j\in \llbracket-sn,sn -1 \rrbracket$, $f_j(p_j^s) = p_{j+1}^s$, and
\begin{align*}
\|p_j^{s}-q_j \|_j &\leq K_\Gamma \sum_{k=-sn+1}^{sn} \| f_{k-1}(q_{k-1})-q_k\|_k \exp(-\lambda_\Gamma|k-j|) \\
&\leq K_\Gamma \sum_{l=1}^{n} \|f_{l-1}(q_{l-1})-q_l\|_l \sum_{h=-s}^{s-1} \exp(-\lambda_\Gamma|l+hn-j|).
\end{align*}
Adding the previous inequality over $j \in \llbracket 0, n-1\llbracket$, we obtain
\begin{align*}
\sum_{j=0}^{n-1} \|p_j^{s}-q_j \|_j  &\leq K_\Gamma \sum_{l=1}^{n} \|f_{l-1}(q_{l-1})-q_l\|_l \sum_{j=1}^{n} \sum_{h=-s-1}^{s-1} \exp(-\lambda_\Gamma|j+hn-l|) \\
&\leq  K_\Gamma \sum_{l=1}^{n} \|f_{l-1}(q_{l-1})-q_l\|_l  \sum_{k=-(s-1)n}^{(s+1)n-1} \exp(-\lambda_\Gamma|l-k|).
\end{align*}
By compactness of the balls $B_j(\frac{\rho}{2})$ one can extract a subsequence over the index $s$ of $(p_j^s)_{j=-sn}^{sn}$ converging for every $j\in\mathbb{Z}$ to a sequence $(p_j)_{j\in\mathbb{Z}}$. Using the estimate 
\[
\sum_{k=-\infty}^{+\infty} \exp(-\lambda_{\Gamma} |k|) = \frac{1+\exp(-\lambda_{\Gamma)}}{1-\exp(-\lambda_{\Gamma})},
\] 
we have for every $j\in\mathbb{Z}$, $f_j(p_j)=p_{j+1}$,
\begin{gather*}
\|p_j-q_j \|_j \leq K_\Gamma \frac{1+\exp(-\lambda_{\Gamma)}}{1-\exp(-\lambda_{\Gamma})} \, \sum_{l=1}^{n} \|f_{l-1}(q_{l-1})-q_l\|_l. 
\end{gather*}
Moreover
\begin{gather*}
\sum_{j=0}^{n-1} \|p_j-q_j \|_j  \leq  K_\Gamma \frac{1+\exp(-\lambda_{\Gamma)}}{1-\exp(-\lambda_{\Gamma})} \, \sum_{l=1}^{n} \|f_{l-1}(q_{l-1})-q_l\|_l, 
\end{gather*}
Let be $\tilde p_j := p_{j+n}$. As $\|\tilde p_j-p_j\|_j$ is uniformly bounded in $j$ and $f_j(\tilde p_j) = \tilde p_{j+1}$, $f_j(p_j) = p_{j+1}$, for every $j$,  the cone property given in Lemma \ref{Lemma:EquivarianceUnstableCone} implies $\tilde p_j = p_j$ for every $j\in\mathbb{Z}$ and therefore $(p_j)_{j\in\mathbb{Z}}$ is a periodic sequence, $p_{j+n}=p_j$ for every  $j\in\mathbb{Z}$.
\end{proof}

The proof of Theorem \ref{Theorem:AnosovShadowingImproved} is done by rewriting a pseudo orbit under the dynamics of $f$ as a pseudo orbit in adapted local charts.

\begin{proof}[\bf Proof of Theorem \ref{Theorem:AnosovShadowingImproved}]
Let $\Gamma_\Lambda=(\Gamma, E, F, A, N)$ be a family of adapted local charts and $(\sigma^u,\sigma^s,\eta,\rho)$ be a  set of hyperbolic constants  as defined in \ref{Definition:AdaptedLocalCharts}. We assume that $\eta$ is chosen as in Theorem \ref{Theorem:AdaptedAnosovShadowingLemma}. We define $\Omega = \cup_{x\in\Lambda} \gamma_x(B_x(\rho))$, we denote by $\Lip(f)$ the Lipschitz constant of $f$ over $\Omega$, by $\Lip(\Gamma_\Lambda)$ the supremum of $\Lip_x(\gamma_x)$ and $\Lip_x(\gamma_x^{-1})$ over $x \in \Lambda$ with respect to the adapted norm $\| \cdot \|_x$. Let  
\[
\epsilon_{AS} := \frac{\epsilon(\rho)}{(1+\Lip(\Gamma_\Lambda))^2(1+\Lip(f))}.
\]
Let  $\Omega_{AS} =\cup_{x'\in\Lambda} \gamma_{x'}(B_{x'}(\epsilon_{AS}))$ and $(x_i)_{i=0}^{n}$ be an $\epsilon_{AS}$-pseudo orbit in $\Omega_{AS}$. Let $(x'_i)_{i=0}^n$ be a sequence of points in $\Lambda$ such that  $x_i \in \gamma_{x'_i}(B_{x'_i}(\epsilon_{AS}))$. Then
\begin{gather*}
d(f(x_i'),f(x_i)) \leq \Lip(f) d(x'_i,x_i) \leq \Lip(f) \Lip(\Gamma_\Lambda) \epsilon_{AS}, \\
d(f(x_i),x_{i+1}) \leq \epsilon_{AS}, \\
d(x_{i+1},x'_{i+1}) \leq \Lip(\Gamma_\Lambda) \epsilon_{AS}, 
\intertext{which implies} 
d(f(x'_i),x'_{i+1}) \leq \left[\Lip(\Gamma_\Lambda)(1+\Lip(f)) + 1 \right]\epsilon_{AS} \leq \epsilon(\rho)/(1+\Lip(\Gamma)), \\
d(f(x_i),x'_{i+1}) \leq (1+\Lip(\Gamma_\Lambda)) \epsilon_{AS} \leq \epsilon(\rho)/(1+\Lip(\Gamma_\Lambda)), \\
f(x_i),f(x'_i) \in  \gamma_{x'_{i+1}}(B_{x'_{i+1}}(\epsilon(\rho))).
\end{gather*}
We have proved that, $\forall\, i \in \llbracket0,n-1\rrbracket, \ x'_i \overset{\Gamma_\Lambda}{\to}x'_{i+1}$ is an admissible transition. Let  $q_i \in B_{x'_i}(\epsilon_{AS})$ such that $\gamma_{x'_i}(q_i) = x_i$. Then $q_i \in B_{x'_i}(\frac{\rho}{2})$ and $f_{x'_i,x'_{i+1}}(q_i) \in B_{x'_{i+1}}(\frac{\rho}{2})$.

Let   $E_i^{u,s} = E_{x'_i}^{u,s}$, $\| \cdot \|_i = \| \cdot \|_{x'_i}$, $f_i := f_{x'_i,x'_{i+1}}=\gamma_{x'_{i+1}}^{-1} \circ f \circ \gamma_{x'_i}$, $A_i = A_{x'_i,x'_{i+1}}$, then $( f_i,A_i, E_i^{u/s}, \| \cdot \|_i)$ satisfies the hypothesis of Theorem \ref{Theorem:AdaptedAnosovShadowingLemma}. There exists a sequence $(p_i)_{i=0}^n$ of points $p_i \in B_{x'_i}(\rho)$ such that for every $i \in \llbracket 0, n-1 \rrbracket$, $f_{x'_i,x'_{i+1}}(p_i) =p_{i+1}$, and for every $i \in \llbracket 0, n \rrbracket$,
\begin{gather*}
\|q_i - p_i \|_{x'_i} \leq K_\Gamma \sum_{k=1}^{n} \|f_{x'_{k-1},x'_{k}}(q_{k-1}) - q_{k} \|_{x'_k} \exp(-\lambda_\Gamma|k-i|), \\
\sum_{i=0}^n \|q_i-p_i\|_{x'_i} \leq  K_\Gamma \sum_{k=1}^{n}  \|f_{x'_{k-1},x'_{k}}(q_{k-1}) - q_{k} \|_{x'_k}, \\
\max_{0 \leq i \leq n}  \|q_i-p_i\|_{x'_i} \leq K_\Gamma \max_{1 \leq k \leq n}  \|f_{x'_{k-1},x'_{k}}(q_{k-1}) - q_{k} \|_{x'_k}.
\end{gather*}
We conclude the proof by taking $y= \gamma_{x'_0}(p_0)$, 
\[
K_{AS} = \Lip(\Gamma_\Lambda)^2K_{\Gamma} \ \ \text{and} \ \ \lambda_{AS} = \lambda_\Gamma. \qedhere
\]
\end{proof}

Using the second part of Theorem \ref{Theorem:AdaptedAnosovShadowingLemma}, we improve the Anosov shadowing property for periodic pseudo orbits (instead of pseudo orbits). 

\begin{proposition}[Anosov periodic shadowing lemma] \label{Proposition:AnosovPeriodicShadowingLemma}
Let $(M,f)$ be a $C^1$ dynamical system  and $\Lambda \subseteq M$ be a locally maximal hyperbolic set. Then there exists a constant  $K_{APS}\geq1$ such that for every $n\geq1$, for every periodic $\epsilon_{AS}$-pseudo orbit $(x_i)_{0 \leq i \leq n}$  of the neighborhood $\Omega_{AS} := \{ x \in M : d(x,\Lambda) < \epsilon_{AS} \}$,  there exists a periodic point $p \in \Lambda$ of period $n$ such that
\begin{gather}
\sum_{i=1}^{n}d(x_i,f^i(p)) \leq K_{APS} \sum_{k=1}^n d(f(x_{k-1}),x_k), \\
\max_{0 \leq i \leq n-1} d(x_i,f^i(p)) \leq K_{APS} \max_{0 \leq k \leq n-1} d(f(x_{k}),x_{k+1}),
\end{gather}
where  $K_{APS} =  K_{AS} \frac{1+\exp(-\lambda_{AS})}{1-\exp(-\lambda_{AS})}$, and $\epsilon_{AS}$, $K_{AS}$,  $\lambda_{AS}$, are the constants given in Theorem  \ref{Theorem:AnosovShadowingImproved}.
\end{proposition}

\begin{proof}
The proof is similar to the proof of Theorem \ref{Theorem:AnosovShadowingImproved}. We will not repeat it.
\end{proof}

\section{The discrete Lax-Oleinik operator}
\label{section:DiscreteLaxOleinikOperator}

We extend the definition of the Lax-Oleinik operator for bijective or not bijective maps and show how  Bousch's approach  helps us to construct  a  subaction (item \ref{item02:BasicErgodicOptimization} of Definition \ref{Definition:DiscreteErgodicOptimization}). We actually construct  a calibrated subaction as explained below that is a stronger notion.

\begin{definition}[Discrete Lax-Oleinik operator] \label {definition:LaxOleinikOperator}
Let $(M,f)$ be a topological dynamical system, $\Lambda \subseteq M$ be a compact $f$-invariant subset, $\Omega \supset \Lambda$ be an open neighborhood of $\Lambda$ of compact closure, and $\phi \in  C^0(\bar\Omega, \mathbb{R})$. Let  $C\geq0$ be a nonnegative constant, and $\bar \phi_\Lambda$ be the ergodic minimizing value of the restriction $\phi$ to  $\Lambda$, see \eqref{Equation:ErgodicMinimizingValue}.  
\begin{enumerate}
\item The {\it Discrete Lax-Oleinik operator} is the nonlinear operator $T$ acting on the space of  functions $u : \bar\Omega \to \mathbb{R}$ defined by
\begin{gather} \label{equation:WeakKAMSolution}
\forall x' \in \bar\Omega, \ T[u](x') := \inf_{x \in \bar\Omega} \big\{ u(x) + \phi(x) -\bar \phi_\Lambda + C d(f(x),x') \big\}.
\end{gather}
\item A {\it calibrated subaction of the Lax-Oleinik operator} is a continuous  function  $u : \bar\Omega \to \mathbb{R}$ solution of the equation
\begin{equation} \label{Equation:CalibratedSubaction}
T[u] = u.
\end{equation}
\end{enumerate}
\end{definition}

The Lax-Oleinik operator is a fundamental tool for studying the set of minimizing configurations in ergodic optimization (Thermodynamic formalism) or discrete Lagrangian dynamics (Aubry-Mather theory, weak KAM theory), see for instance \cite{GaribaldiThieullen2011,Garibaldi2017,SuThieullen2018,Jenkinson2019}. A calibrated subaction is in some sense an optimal subaction. For expanding endomorphisms or one-sided subshifts of finite type, the theory is well developed, see for instance Definition 3.A in Garibaldi \cite{Garibaldi2017}. Unfortunately the standard definition requires the existence of many inverse branches. Definition  \ref{definition:LaxOleinikOperator} is new and valid for two-sided subshifts of finite type and more generally for hyperbolic systems as in the present paper.

Following Bousch's approach, we define the following criteria.  A similar notion for flows can be introduced, see \cite{SuThieullen2021}.

\begin{definition}[Discrete   positive Liv\v{s}ic  criteria] \label{Definition:DiscretePositiveLivsicCriteria}
Let  $(M,f,\phi,\Lambda,\Omega,C)$ be  as in Definition \ref{definition:LaxOleinikOperator}. We say that $\phi$ satisfies the {\it discrete positive Liv\v{s}ic  criteria on $\Omega$ with distortion constant $C$} if
\begin{gather} \label{equation:BouschCriteria}
\inf_{n\geq1} \ \inf_{(x_0,x_1,\ldots,x_n) \in \bar\Omega^{n+1}} \ \sum_{i=0}^{n-1} \big( \phi(x_i) -\bar\phi_\Lambda + C d(f(x_i),x_{i+1}) \big) > -\infty.
\end{gather}
\end{definition} 
The discrete positive Liv\v{s}ic  criteria is the key ingredient  of the proof of the existence of a  calibrated subaction with a controlled Lipschitz constant. Here $\Lip(\phi)$, $\Lip(u)$, denote the Lipschitz constant of $\phi$ and $u$ restricted on $\bar\Omega$ respectively.

\begin{proposition} \label{proposition:WeakKAMsolution}
Let   $(M,f,\phi,\Lambda,\Omega,C)$  be as in Definition \ref{definition:LaxOleinikOperator}. Assume that $\phi$ satisfies the discrete positive Liv\v{s}ic  criteria. Then  
\begin{enumerate}
\item the Lax-Oleinik operator admits a $C^0$ calibrated subaction,  
\item every $C^0$ calibrated subaction  $u$ is  Lipschitz with $\Lip(u) \leq C$.
\end{enumerate}
\end{proposition}

Notice that conversely the discrete positive Liv\v{s}ic criteria is satisfied whenever $\phi$ admits a  Lipschitz subaction $u$ with $\Lip(u) \leq C$. When $C=0$ and the infimum in \eqref{equation:BouschCriteria} is taken over true orbits instead of all sequences, there always exists a lower semi-continuous subaction \eqref{Equation:Subaction} as it is discussed in \cite{SuThieullen2019}.

We recall without proof some basic facts of the Lax-Oleinik operator.

\begin{lemma}
Let $T$ be the Lax-Oleinik operator as in Definition~\ref{definition:LaxOleinikOperator}. Then
\begin{enumerate}
\item if $u_1 \leq u_2$ then $T[u_1] \leq T[u_2]$,
\item for every constant $c \in\mathbb{R}$, $T[u+c] = T[u]+c$,
\item for every sequence of functions $(u_n)_{n\geq0}$ bounded from below, 
\[
T[\inf_{n\geq0} u_n] = \inf_{n\geq0} T[u_n].
\]
\end{enumerate}
\end{lemma}

\begin{proof}[Proof of Proposition \ref{proposition:WeakKAMsolution}]
Define
\[
\forall\, x,y \in \bar\Omega, \ E(x,y) := \phi(x)-\bar\phi_\Lambda +Cd(f(x),y),
\]
and
\[
I := \inf_{n\geq1} \ \inf_{(x_0,x_1,\ldots,x_n) \in \bar\Omega^{n+1}} \ \sum_{i=0}^{n-1} E(x_i,x_{i+1}).
\]

{\it Part 1.} We show that $T[u]$ is $C$-Lipschitz whenever $u$ is continuous. Indeed if $x',y' \in \bar\Omega$ are given,
\begin{alignat*}{2}
T[u](x') &= u(x) + E(x,x'), &\quad &\text{for some $x\in \bar\Omega$,} \\
T[u](y') &\leq u(y) + E(y,y'), &\quad &\text{for every $y \in \bar\Omega$.}
\end{alignat*}
Then by choosing $y=x$ in the previous inequality, we obtain
\begin{gather*}
T[u](y') -T[u](x') \leq E(x,y')-E(x,y) = C \big[ d(f(x),y')-d(f(x),y) \big] \leq Cd(y',y).
\end{gather*}

{\it Part 2.} Let $v := \inf_{n\geq 0} T^n[0]$. We show that $v$ is $C$-Lipschitz, non positive,  and satisfies $T[v] \geq  v$. Indeed we first have
\[
\forall n\geq1, \ \forall x' \in \bar\Omega, \ T^n[0](x') = \inf_{x_0,\ldots,x_n=x'} \ \sum_{i=0}^{n-1} E(x_i,x_{i+1}) \geq I.
\]
Moreover $v$ is  $C$-Lipschitz since $T^n[0]$ is $C$-Lipschitz thanks to part 1. Finally we have
\[
T[v] = T[\inf_{n\geq 0}  T^n[0] ] =  \inf_{n\geq0}  T^{n+1}[v]  \geq  v.
\]

{\it Part 3.} Let $u := \sup_{n\geq0} T^n[v]  = \lim_{n\to+\infty} T^n[v]$. We show that $u$ is a $C$-Lipschitz calibrated subaction. We already know from parts 1 and 2 that $T^n[v]$ is $C$-Lipschitz for every $n\geq0$. Using the definition of $\bar \phi_\Lambda$, we know that, for every $n\geq1$ there exists $x \in \Lambda$ such that $\sum_{n=0}^{n-1} \big( \phi \circ f^i (x) -\bar \phi_\Lambda \big) \leq 0$, and using the fact that $T^n[v]$ is $C$-Lipschitz, we have 
\begin{gather*}
T^n[v](f^n(x)) \leq  v(x) + \sum_{i=0}^{n-1} E(f^i(x),f^{i+1}(x)) = v(x) +\sum_{k=0}^{n-1}(\phi \circ f^k(x) -\bar\phi_\Lambda) \leq 0, \\
T^n[v](x') \leq C d(x',f^n(x)) \leq  C \text{\rm diam}(\bar\Omega), \quad \forall x' \in \bar\Omega.
\end{gather*}
Since $T[v] \geq v$, we also have $T[u] \geq u$. We next show $T[u] \leq u$. Let $x' \in \bar\Omega$ be given. For every $n\geq1$,  $T[T^n[v]] = T^{n+1}[v] \leq u$, there exists $x_n \in \bar\Omega$ such that 
\[
T^n[v](x_n) + E(x_n,x') \leq u(x').
\]
By compactness of $\bar\Omega$, $(x_n)_{n\geq1}$ admits a converging subsequence (denoted the same way) to some $x_\infty\in \bar\Omega$. Thanks to the uniform Lipschitz constant of the sequence $(T^n[v])_{n\geq1}$ and the fact that $\lim_{n\to+\infty}T^n[v]=u$, we obtain,
\[
\forall\, x' \in\bar\Omega, \ T[u](x') = \inf_{x \in \bar\Omega} \{ u(x)+E(x,x') \} \leq u(x_\infty) + E(x_\infty,x') \leq u(x').
\]
We have proved $T[u]=u$ and $u$ is $C$-Lipschitz.
\end{proof}

\section{The discrete positive Liv\v{s}ic criteria}
\label{section:ProofDiscreteCase}

Let $(M,f)$ be a $C^1$ dynamical system, $\Lambda \subseteq M$ be a locally maximal hyperbolic compact subset, and $\phi :M \to \mathbb{R}$ be a Lipschitz continuous function. A calibrated subaction $u$ \eqref{Equation:CalibratedSubaction} is in particular a subaction \eqref{Equation:Subaction}
\[
\forall x \in \bar\Omega,\quad u\circ f(x) -u(x) \leq \phi(x) - \bar \phi_\Lambda. 
\]
Theorem  \ref{Theorem:DiscreteSubactionExistence} is therefore a consequence of Proposition \ref{proposition:WeakKAMsolution} provided we prove that $f$ satisfies the discrete positive Liv\v{s}ic criteria \eqref{equation:BouschCriteria}. 

\begin{proposition} \label{Proposition:ValidityDiscreteCriteria}
Let  $(M,f,\phi,\Lambda,\Omega,C)$ be as in Definition \ref{definition:LaxOleinikOperator}. Then $\phi$ satisfies the discrete positive Liv\v{s}ic criteria.
\end{proposition}

For a true orbit instead of a pseudo orbit, the criteria  amounts to bounding from below the normalized Birkhoff sum $\frac{1}{n}\sum_{i=0}^{n-1} \big( \phi \circ f^i(x) - \bar \phi \big)$. As we saw in \cite{SuThieullen2019}, this is equivalent to the existence of  a bounded lower semi-continuous subaction. To obtain a better regularity of the subaction we need the stronger criteria \eqref{equation:BouschCriteria}.

We first start by proving two intermediate lemmas, Lemma \ref{lemma:FinitenessPeriodicPseudoBirkhoffSums} for periodic pseudo-orbits, and Lemma \ref{lemma:FinitenessPseudoBirkhoffSums} for pseudo-orbits. Denote
\[
\Omega(\epsilon)  := \{ x \in M : d(x,\Lambda) < \epsilon\}.
\]
We recall that $\epsilon_{AS}$, $\Omega_{AS} =  \Omega(\epsilon_{AS})$, and $K_{APS}$, have been defined in Theorem \ref{Theorem:AnosovShadowingImproved} and Proposition \ref{Proposition:AnosovPeriodicShadowingLemma}.

\begin{lemma} \label{lemma:FinitenessPeriodicPseudoBirkhoffSums}
Let  $C \geq K_{APS} \Lip(\phi)$. Then for every periodic $\epsilon_{AS}$-pseudo orbit $(x_i)_{i=0}^n$  of  $\Omega_{AS}$,
\[
\sum_{i=0}^{n-1} \big( \phi (x_i) -\bar \phi_\Lambda +C d(f(x_i),x_{i+1}) \big) \geq 0.
\]
\end{lemma}

\begin{proof}
Proposition \ref{Proposition:AnosovPeriodicShadowingLemma} tells us that there exists a periodic orbit $p\in\Lambda$, $f^n(p)=p$, such that
\[
\sum_{i=0}^{n-1} d(f(x_i,f^i(p) ) \leq  K_{APS} \sum_{i=0}^{n-1} d(f(x_i) ,x_{i+1}).
\]
Then
\begin{align*}
\sum_{i=0}^{n-1} \big( \phi(x_i) &- \bar \phi_\Lambda +C d(f(x_i) ,x_{i+1}) \big) \\
&\geq \sum_{i=0}^{n-1} \big( \phi \circ f^i(p) - \bar\phi_\Lambda \big) + \sum_{i=0}^{n-1} \big( \phi(x_i) - \phi \circ f^i(p) +C d(f(x_i) ,x_{i+1}) \big)  \\
&\geq  \sum_{i=0}^{n-1} \big( \phi \circ f^i(p) - \bar\phi_\Lambda \big) + \sum_{i=0}^{n-1}\big( - \Lip(\phi) d(x_i,f^i(p)) + C d(f(x_i) ,x_{i+1}) \big) \\
&\geq  \sum_{i=0}^{n-1} \big( \phi \circ f^i(p) - \bar\phi_\Lambda \big) \geq 0. \qedhere
\end{align*}
\end{proof}

\begin{lemma} \label{lemma:PseudoPeriodicSegment}
Let  $N_\epsilon\geq1$ be the smallest number of balls of radius $\epsilon/2$ that  can cover $\Omega_{\epsilon}$. Let $(x_i)_{i=0}^{n}$ be a sequence of points of  $\Omega_{\epsilon}$. Then there exists $r \in \llbracket 1, N_\epsilon \rrbracket$ and times $0=\tau_0 < \tau_1 <  \cdots  < \tau_{r} =  n$ such that, 
\begin{enumerate}
\item $\forall\, k \in \llbracket 1,r-1\rrbracket, \ \forall\, l \in \llbracket 0, k-1 \rrbracket, \ \forall\, j \in  \llbracket \tau_{k}, n-1\rrbracket, \ d(x_{j},x_{\tau_l}) \geq \epsilon$,
\item $\forall\, k \in \llbracket 1,r-1\rrbracket$, if $\tau_{k} \geq \tau_{k-1}+2$  then $d(x_{\tau_{k}-1}, x_{\tau_{k-1}}) < \epsilon$,
\item either $d(x_{\tau_{r}-1}, x_{\tau_{r-1}}) < \epsilon$ or $d(x_{\tau_r},x_{\tau_{r-1}})<\epsilon$.
\end{enumerate}
\end{lemma}

\begin{figure}[hbt]
\centering
\includegraphics[width=\textwidth]{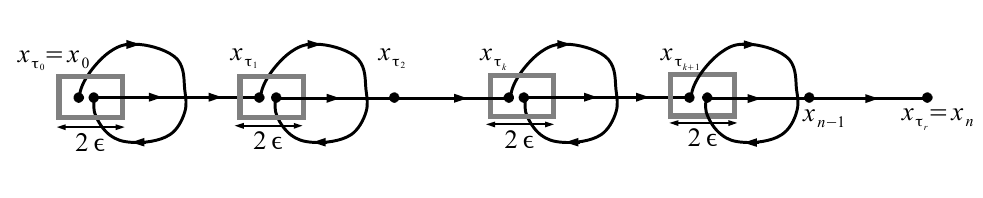}
\caption{The schematic $r$ returns of Lemma \ref{lemma:PseudoPeriodicSegment}.}
\end{figure}

\begin{proof}
We construct by induction the sequence $\tau_k$. Assume we have constructed $\tau_{k} < n$. Define
\[
T := \{ j \in \llbracket \tau_k+1, n\rrbracket : d(x_j,x_{\tau_k}) <   \epsilon \}.
\]
If $T = \emptyset$, choose $\tau_{k+1} = \tau_k +1$; if $T \not=\emptyset$ and $\max(T) <n$ then $\tau_{k+1} = \max(T)+1$, $d(x_{\tau_{k+1}-1},x_{\tau_k}) <  \epsilon$ and for every $j \geq \tau_{k+1}$, $d(x_j, x_{\tau_k}) \geq \epsilon$; if $\max(T)=n$ then $\tau_{k+1}=n$. Since $(x_{\tau_k})_{k=0}^{r-1}$ are $\epsilon$ apart, $r \leq N_\epsilon$.
\end{proof}

\begin{lemma} \label{lemma:FinitenessPseudoBirkhoffSums}
Let $C = K_{APS} \Lip(\phi)$ and $N_{AS}$ be the smallest number of balls of radius $\epsilon_{AS}/2$ that  can cover $\Omega_{AS}$. Let $\delta_{AS} := N_{AS} \, \diam(\Omega_{AS})$. Then for every $\epsilon_{AS}$-pseudo orbit $(x_i)_{i=0}^n$  of  $\Omega_{AS}$,
\[
\sum_{i=0}^{n-1} \big( \phi (x_i) -\bar \phi_\Lambda +C d(f(x_i),x_{i+1}) \big) \geq -\Lip(\phi)\delta_{AS} .
\]
\end{lemma}

\begin{proof}
We split the pseudo orbit $(x_i)_{i=0}^{n-1}$ into $r \leq N_{AS}$ segments of the form $(x_i)_{i=\tau_k}^{\tau_{k+1}-1}$   according to Lemma \ref{lemma:PseudoPeriodicSegment}, for $0 \leq k \leq r-1$ with $0 =\tau_0 < \tau_1 < \cdots < \tau_r=n$.  To simplify the notations, denote
\[
\phi_i := \phi (x_i) -\bar\phi_\Lambda +C d(f(x_i),x_{i+1}).
\]
Notice that for every $i \in \llbracket 0,n-1\rrbracket$
\[
\phi_{i} \geq -\Lip(\phi) \diam(\Omega_{AS}).
\] 
If $\tau_{k+1} \geq \tau_k +2$ and $k \in \llbracket 0, r-1\llbracket$ then $d(x_{\tau_k},x_{\tau_{k+1}-1}) < \epsilon_{AS}$, $(x_i)_{i=\tau_k}^{\tau_{k+1}-1}$ is a periodic  pseudo orbit as in Lemma  \ref{lemma:FinitenessPeriodicPseudoBirkhoffSums} and
\[
\sum_{i=\tau_k}^{\tau_{k+1}-2} \phi_i \geq 0, \quad \sum_{i=\tau_k}^{\tau_{k+1}-1} \phi_i \geq -\Lip(\phi) \diam(\Omega_{AS}).
\]
If $\tau_r \geq \tau_{r-1}+2$ then either $(x_i)_{i=\tau_{r-1}}^{\tau_r-1}$ or $(x_i)_{i=\tau_{r-1}}^{\tau_r}$ is a periodic  pseudo orbit.  In both cases we have
\[
\sum_{i=\tau_{r-1}}^{\tau_{r}-1} \phi_i \geq -\Lip(\phi) \diam(\Omega_{AS}).
\]
If $\tau_{k+1} =  \tau_k +1$ then
\[
\sum_{i=\tau_k}^{\tau_{k+1}-1} \phi_i = \phi_{\tau_k} \geq -\Lip(\phi) \diam(\Omega_{AS}).
\]
By adding these inequalities for $k \in \llbracket 0, r-1 \rrbracket$, we have
\[
\sum_{i=\tau_0}^{\tau_{r}-1} \phi_i \geq -\Lip(\phi) N_{AS}\diam(\Omega_{AS}). \qedhere
\]
\end{proof}

We recall that $K_{APS}$, $\epsilon_{AS}$, have been defined in Theorem \ref{Theorem:AnosovShadowingImproved}, Proposition \ref{Proposition:AnosovPeriodicShadowingLemma}, and $N_{AS}$, $\delta_{AS}$, in Lemma \ref{lemma:FinitenessPseudoBirkhoffSums}.

\begin{proof}[\bf Proof of Proposition \ref{Proposition:ValidityDiscreteCriteria}]
Let  $(x_i)_{i=0}^n$ be a sequence of points of $\Omega_{AS}$. We split the sequence into disjoint segments $(x_i)_{i=\tau_k}^{\tau_{k+1}-1}$, $0= \tau_0 < \tau_1 < \cdots < \tau_k < \tau_{k+1} < \cdots  < \tau_r=n$, having one of the following form.

{\it Segment of the first kind:} $\tau_{k+1}= \tau_k + 1$ and $d(f(x_{\tau_k}),x_{\tau_{k+1}}) \geq \epsilon_{AS}$. Then
\begin{gather*}
\phi(x_{\tau_k}) - \bar \phi_\Lambda \geq -\Lip(\phi) \diam(\Omega_{AS}), \quad d(f(x_{\tau_k}),x_{\tau_k + 1}) \geq \epsilon_{AS}.
\end{gather*}
By choosing $C \geq \Lip(\phi)\diam(\Omega_{AS})/\epsilon_{AS}$, we obtain
\[
\phi(x_{\tau_k}) - \bar \phi_\Lambda + C  d(f(x_{\tau_k}),x_{\tau_k + 1})\geq0.
\]

{\it Segment of the second kind:} $\tau_{k+1} \geq  \tau_k + 2$ and
\[
\left\{\begin{array}{l}
\forall\, \tau_k \leq i \leq \tau_{k+1}-2, \ d(f(x_i),x_{i+1}) < \epsilon_{AS},  \\
d(f(x_{\tau_{k+1}-1}), x_{\tau_{k+1}}) \geq \epsilon_{AS}.
\end{array}\right.
\]
Then $(x_i)_{i=\tau_k}^{\tau_{k+1}-1}$ is a pseudo orbit. By using Lemma \ref{lemma:FinitenessPseudoBirkhoffSums} and $C \geq K_{APS} \Lip(\phi)$, we have
\begin{gather*}
\sum_{i=\tau_k}^{\tau_{k+1}-2} \big( \phi (x_i) -\bar \phi_\Lambda +C d(f(x_i),x_{i+1}) \big) \geq -\Lip(\phi)\delta_{AS}, \\
\phi(x_{\tau_{k+1}-1}) - \bar \phi_\Lambda + C d(f(x_{\tau_{k+1}-1}),x_{\tau_{k+1}}) \geq -\Lip(\phi) \diam(\Omega_{AS}) + C \epsilon_{AS}.
\end{gather*}
By choosing $C \geq \Lip(\phi) (\delta_{AS}+\diam(\Omega_{AS}))/\epsilon_{AS}$, we obtain
\[
\sum_{i=\tau_k}^{\tau_{k+1}-1} \big( \phi(x_i) - \bar \phi_\Lambda + C d(f(x_i), x_{i+1})  \big) \geq 0.
\]

{\it Segment of the third kind:} if it exists, this segment is the last one and $(x_i)_{i=\tau_{r-1}}^{\tau_r}$ is a pseudo orbit. By using again Lemma \ref{lemma:FinitenessPseudoBirkhoffSums} 
\[
\sum_{i=\tau_{r-1}}^{\tau_r-1} \big( \phi (x_i) -\bar \phi_\Lambda +C d(f(x_i),x_{i+1}) \big) \geq -\Lip(\phi)\delta_{AS}.
\]
Notice that we can choose $K_\Lambda := \max( K_{APS}, (N_{AS}+1)  \diam(\Omega_{AS})/\epsilon_{AS})$ in Theorem~\ref{Theorem:DiscreteSubactionExistence}.
\end{proof}

\begin{proof}[\bf Proof of Theorem \ref{Theorem:DiscreteSubactionExistence}]
The proof readily follows from the conclusions of  Propositions \ref{proposition:WeakKAMsolution} and \ref{Proposition:ValidityDiscreteCriteria}.
\end{proof}

\vfill
\pagebreak
\appendix
\appendixpage
\addappheadtotoc

\section{Local hyperbolic dynamics}
\label{Appendix:LocalHyperbolicDynamics}

We recall in this section the local theory of hyperbolic dynamics. The dynamics is obtained by iterating a sequence of (non linear) maps defined locally and close to uniformly hyperbolic linear maps.  The notion of adapted local charts is defined in \ref{section:AdaptedLocalCharts}. In these charts the expansion along the unstable direction, or the contraction along the stable direction, is realized at the first iteration, instead of after some number of iterations. It is a standard notion that can be extended in different directions, see for instance Gourmelon \cite{Gourmelon2007}.

\subsection{Adapted local hyperbolic  map}

We recall in this section the notion of local hyperbolic maps. The constants $(\sigma^s,\sigma^u,\eta,\rho)$ that appear in the following definition are used in the proof of Theorem \ref{Theorem:AdaptedAnosovShadowingLemma}.

\begin{definition}[Adapted local hyperbolic map] \label{Definition:AdaptedLocalHyperbolicMap}
Let $(\sigma^s,\sigma^u,\eta,\rho)$ be positive real numbers called {\it constants of hyperbolicity}.  Let $\mathbb{R}^d = E^u \oplus E^s $ and $\mathbb{R}^d= \tilde E^u \oplus \tilde E^s$ be two Banach spaces equiped with two norms $| \cdot|$ and $\| \cdot \|$ respectively. Let $P^u : \mathbb{R}^d \to E^u$ and $P^s : \mathbb{R}^d \to E^s$ be the two linear projectors associated with the splitting $\mathbb{R}^d =  E^u \oplus  E^s$ and similarly  $\tilde P^u : \mathbb{R}^d \to \tilde E^u$ and $\tilde P^s : \mathbb{R} \to \tilde E^s$ be the two projectors associated with  $\mathbb{R}^d = \tilde E^u \oplus \tilde E^s$. Let $B(\rho)$, $B^u(\rho)$, $B^s(\rho)$ be the balls of radius $\rho$ on each $E, E^u, E^s$ respectively, with respect to the norm $| \cdot |$. Let $\tilde B(\rho)$, $\tilde B^u(\rho)$,  $\tilde B^s(\rho)$ be the corresponding balls with respect to the norm $\| \cdot \|$. We assume that both norms are {\it sup norm adapted to the splitting} in the sense,
\[
\left\{ \begin{array}{l}
\forall v,w \in E^u \times E^s, \ |v+w| = \max(|v|,|w|), \\
\forall v,w \in \tilde E^u \times \tilde E^s, \ \|v+w\| = \max(\|v\|, \|w\|).
\end{array}\right.
\]
In particular $B(\rho) = B^u(\rho) \times B^s(\rho)$, $\tilde{B}(\rho) = \tilde{B}^u(\rho) \times \tilde{B}^s(\rho)$. We also  assume
\[
\sigma^u > 1 > \sigma^s, \quad \eta < \min\Big( \frac{\sigma^u-1}{6},\frac{1-\sigma^s}{6} \Big), \quad \epsilon(\rho) := \rho \min\Big( \frac{\sigma^u-1}{2},\frac{1-\sigma^s}{8}\Big).
\]
An {\it adapted local hyperbolic map with respect to the two norms and the constants of hyperbolicity } is a set of data $(f,A, E^{u/s}, \tilde E^{u/s}, |\cdot|,\| \cdot\|)$ such that:
\begin{enumerate}
\item $f : B(\rho) \to \mathbb{R}^d$ is a Lipschitz map,
\item  $A:\mathbb{R}^d \to \mathbb{R}^d$ is a linear map which may not be invertible and is defined into block matrices
\[
A = \begin{bmatrix}
A^u & D^u \\ D^s &  A^s
\end{bmatrix}, \quad 
\left\{\begin{array}{l}
(v,w) \in E^u \times E^s, \\
A(v+w) = \tilde v + \tilde w,
\end{array}\right.
\ \Rightarrow \
\left\{\begin{array}{l}
\tilde v = A^u v + D^u w \in \tilde E^u, \\ \tilde w = D^s v  + A^s w \in \tilde E^s,
\end{array}\right.
\]
that satisfies
\[
\left\{\begin{array}{l}
\forall \, v \in E^u, \ \|A^u v\| \geq \sigma^u \|v\|, \\
\forall \, w \in E^s, \ \|A^s w\| \leq \sigma^s \|w\|,
\end{array}\right. \quad\text{and}\quad
\left\{\begin{array}{ll}
\|D^u\| \leq \eta, & \Lip(f-A) \leq \eta, \\ \|D^s\| \leq \eta, & \|f(0)\| \leq \epsilon(\rho),
\end{array}\right.
\]
where  the $\Lip$ constant is computed using the two norms $| \cdot |$ and $\| \cdot \|$.
\end{enumerate}
\end{definition}

The constant $\sigma^u$ is called the {\it expanding constant}, $\sigma^s$ is called the {\it contracting constant}. The constant $\rho$ represents a uniform size of local charts. The constant $\epsilon(\rho)$ represents the error in a pseudo-orbit. The constant $\eta$ represents a deviation from the linear map and should be thought of as small compared to the gaps $\sigma^u-1$ and $1-\sigma^s$. Notice that $\epsilon(\rho)$ is independent of $\eta$. The map $f:B(\rho) \to \mathbb{R}^d$ should be considered as a perturbation of its linear part $A$.

\subsection{Adapted local graph transform}

The graph transform is a perturbation technique of a hyperbolic linear map. A hyperbolic linear map preserves a splitting into an unstable vector space on which the linear map is expanding, and a stable vector space on which  the linear map is contracting. We show that a Lipschitz map close to  a hyperbolic linear map  also preserves similar objects that are Lipschitz graphs tangent to the unstable or stable direction. The operator $A$ may have a non trivial kernel, and we don't assume $f$ to be invertible.

\begin{definition} \label{Definition:LocalLipschitGraph}
Let  $(\sigma^u,\sigma^s,\eta,\rho)$, $\mathbb{R}^d =  E^u \oplus E^s  = \tilde E^u \oplus \tilde E^s$ be as in Definition \ref{Definition:AdaptedLocalHyperbolicMap}.
We denote by $\mathcal{G}^u$ the set of Lipschitz graphs over the unstable direction $E^u$ with controlled Lipschitz constant and height. More precisely
\[
\mathcal{G}^u = \Big\{ [G : B^u(\rho) \to B^s(\rho)] : \Lip(G) \leq \frac{6\eta}{\sigma^u-\sigma^s}, \ |G(0)| \leq \frac{\rho}{2} \Big\}.
\]
We denote similarly by $\tilde{\mathcal{G}}^u$ the set of Lipschitz graphs
\[
\tilde{\mathcal{G}}^u := \Big\{ [\tilde G : \tilde B^u(\rho) \to \tilde B^s(\rho)] : \Lip(\tilde G) \leq \frac{6\eta}{\sigma^u-\sigma^s}, \ \|\tilde G(0)\| \leq \frac{\rho}{2} \Big\}.
\]
The graph of $G \in \mathcal{G}^u$ is the subset of $B(\rho)$:
\[
\Graph(G) := \{ v+G(v) : v \in B^u(\rho) \}.
\]
\end{definition}

Notice that $\Lip(G),\Lip(\tilde G) \leq \frac{1}{2}$ for every $(G,\tilde G) \in \mathcal{G}^u \times \tilde{\mathcal{G}}^u$, thanks to the assumptions on $\eta$. Notice also that the Lipschitz constant of $G$ goes to zero as $f$ becomes more and more linear, as  $\eta \to 0$, independently of the location of $f(0)$ controlled by $\epsilon(\rho)$ depending only on $(\sigma^u,\sigma^s, \rho)$. 

\begin{proposition}[Forward local graph transform] \label{Proposition:ForwardLocalGraphTransform}
Let  $(\sigma^u,\sigma^s,\eta,\rho,\epsilon)$, $\mathbb{R}^d = E^u \oplus E^s  = \tilde E^u \oplus \tilde E^s$, and $(A,f)$ be as defined in \ref{Definition:AdaptedLocalHyperbolicMap}. Then
\begin{enumerate}
\item For every graph $G \in \mathcal{G}^u$ there exists a unique graph $\tilde G \in \tilde{\mathcal{G}}^u$ such that
\[
\left\{\begin{array}{l}
\forall \, \tilde v \in \tilde B^u(\rho), \ \exists ! \, v\in B^u(\rho), \ \tilde v = \tilde P^u f(v+G(v)), \\
\tilde G(\tilde x) = \tilde P^s f(v+G(v)).
\end{array}\right.
\]

\item for every $G_1,G_2 \in \mathcal{G}^u$ and $\tilde G_1,\tilde G_2$  the corresponding graphs,
\[
\| \tilde G_1 - \tilde G_2 \|_{\infty} \le (\sigma^s+2\eta) \,| G_1 - G_2 |_{\infty}.
\]

\item the map
\[
(\mathcal{T})^u:= 
\left\{\begin{array}{l}
\mathcal{G}^u \to \tilde{\mathcal{G}}^u, \\
G \mapsto \tilde G,
\end{array}\right.
\]
 is called the forward graph transform. 

\item for every $G \in \mathcal{G}^u$, $f( \Graph(G)) \supseteq \Graph(\tilde G)$ ,
\[
\forall\,  q_1, q_2 \in \Graph(G) \cap f^{-1}(\Graph(\tilde G)), \ \
\| f(q_1) -  f(q_2) \| \geq (\sigma^u-3\eta) \, |q_1 - q_2|.
\]
\end{enumerate}
\end{proposition}

For a detailed proof of this proposition we suggest the monography by Hirsch, Pugh, Shub \cite{HirschPughShub1977}.

\subsection{Adapted local charts}
\label{section:AdaptedLocalCharts}

We consider in this section a $C^1$ dynamical systems $(M,f)$ on a manifold $M$ of dimension $d\geq2$ without boundary,   $\Lambda \subseteq M$ a  hyperbolic  $f$-invariant compact set, and $\Omega \supset \Lambda$ an open neighborhood of $\Lambda$ of compact closure. Let  $\lambda^s < 0 < \lambda^u$, $C_\Lambda\geq1$,  and $T_MM  = E^u_\Lambda(x) \oplus E^s_\Lambda(x)$   as in Definition \ref{Definition:LocallyHyperbolicMap}. We show that we can construct a family of local charts well adapted to the hyperbolicity of $\Lambda$. The existence of such a family  depends only on the  continuity of  $x \in \Lambda \mapsto E^u_\Lambda(x)  \oplus E^s_\Lambda(x)$ and the $C^1$ regularity of $f$. 

\begin{definition}[Adapted local charts] \label{Definition:AdaptedLocalCharts}
Let $(M,f)$ be a $C^1$ dynamical system, $U \subseteq M$ be an open set, and $\Lambda \subseteq U$ be an $f$-invariant compact hyperbolic set with constants of hyperbolicity $(\lambda^u,\lambda^s)$. {\it A family of  adapted local charts}  is a set of data  $\Gamma_\Lambda=(\Gamma, E, N, F,A)$ and a set of constants  $(\sigma^u,\sigma^s,\eta,\rho)$   satisfying the following properties: 
\begin{enumerate}
\item \label{item:AdaptedLocalCharts_01} The constants $(\sigma^u,\sigma^s,\eta,\rho)$  are chosen so that,
\begin{gather*}
\exp(\lambda^s) < \sigma^s < 1 < \sigma^u < \exp(\lambda^u) \\ 
\eta < \min \Big( \frac{\sigma^u-1}{6}, \frac{1-\sigma^s}{6} \Big), \quad  \epsilon(\rho) := \rho \min\Big( \frac{\sigma^u-1}{2},\frac{1-\sigma^s}{8}\Big)
\end{gather*}
where $\lambda^u,\lambda^s$ are the constants of hyperbolicity of $\Lambda$ as in Definition \ref{Definition:LocallyHyperbolicMap}. Notice that  $\epsilon(\rho) < \rho/8$.

\item \label{item:AdaptedLocalCharts_02} $\Gamma =(\gamma_x)_{x \in\Lambda}$ is a parametrized family of charts such that for every $x \in\Lambda$, $\gamma_x : B(1) \subset \mathbb{R}^d \to M$ is a diffeomorphism from the unit ball $B(1)$ of $\mathbb{R}^d$  onto an open set in $M$, $\gamma_x(0) = x$, and such that the $C^1$ norm of $\gamma_x, \gamma_x^{-1}$ is uniformly bounded with respect to $x$.

\item \label{item:AdaptedLocalCharts_03} $E = (E_x^{u/s})_{x \in\Lambda}$ is a parametrized family of splitting $\mathbb{R}^d = E_x^u \oplus E_x^s$  obtained by  pull backward of the corresponding splitting on $T_\Lambda M$ by the tangent map $T_0\gamma_x$ at the origin of $\mathbb{R}^d$,
\[
E_x^u = (T_0 \gamma_x)^{-1}E_\Lambda^u(x), \quad E_x^s := (T_0\gamma_x)^{-1}E_\Lambda^s(x),
\]
and by $\Id = P_x^u + P_x^s$,  the corresponding projectors onto $E_x^u,E_x^s$ respectively.

\item \label{item:AdaptedLocalCharts_05} $N:=(\| \cdot \|_x)_{x \in\Lambda}$ is a $C^0$ parametrized family of  norms. The {\it adapted local norm} is a sup norm adapted to the splitting $E_x^u \oplus E_x^s$ that satisfies
\[
\forall\,v \in E_x^u, \ w \in E_x^s, \quad \|v+w \|_x = \max(\|v\|_x,\|w\|_x).
\]
The ball of radius $\rho$ centered at the origin of $\mathbb{R}^d$ is denoted by $B_x(\rho)$.

\item \label{item:AdaptedLocalCharts_06} The constant $\rho$ is chosen so that $\gamma_x(B_x(\rho)) \subset U$ and
\[
\forall\, x,y \in \Lambda, \quad \big[ f(x) \in \gamma_y(B_y(\rho)) \ \Rightarrow \ f(\gamma_x(B_x(\rho)) \subseteq \gamma_y(B(1)) \big].
\]

\item \label{item:AdaptedLocalCharts_07} $F:=(f_{x,y})_{x,y \in\Lambda}$ is a  family of $C^1$ maps $f_{x,y} : B_x(\rho) \to B(1)$ which is parametrized by  couples of  points $(x,y) \in \Lambda$ satisfying  $f(x) \in \gamma_y(B_y(\rho))$. The {\it  adapted local map} is defined by
\[
\forall\, v \in B_x(\rho), \ f_{x,y}(v) := \gamma_y^{-1} \circ f \circ \gamma_x(v).
\]

\item \label{item:AdaptedLocalCharts_08} $A:=(A_{x,y})_{x,y \in\Lambda}$ is the family of tangent maps $A_{x,y} : \mathbb{R}^d \to \mathbb{R}^d$ of $f_{x,y}$ at the origin, that is parametrized by the couples of points $x,y \in \Lambda$ satisfying  $f(x) \in \gamma_y(B_y(\rho))$. Let 
\[
A_{x,y} := Df_{x,y}(0),
\]
where $Df_{x,y}(0)$ denotes the differential map of $v \mapsto f_{x,y}(v)$ at $v=0$.

\item \label{item:AdaptedLocalCharts_09}  For every $x,y \in\Lambda$ satisfying $f(x) \in \gamma_y(B_y(\epsilon))$, the set of data 
\[
(f_{x,y},A_{x,y}, E_x^{u/s},E_y^{u/s},\| \cdot \|_x ,\| \cdot \|_y)
\] 
is an adapted local hyperbolic map with respect to  the constant of hyperbolicity $(\sigma^u,\sigma^s,\eta,\rho)$ as in Definition \ref{Definition:AdaptedLocalHyperbolicMap}. We have
\begin{gather*}
A_{x,y} =\begin{bmatrix} P^u_yA_{x,y}P^u_x & P^u_yA_{x,y}P^s_x \\ P^s_yA_{x,y}P^u_x & P^s_yA_{x,y}P^s_x \end{bmatrix}, \\ \quad \\
\left\{\begin{array}{l}
\forall\, v \in E_x^u, \ \|A_{x,y}v\|_y \geq \sigma^u \|v\|_x, \\
\forall\, v \in E_x^s, \ \|A_{x,y}v \|_y \leq  \sigma^s \|v\|_x,
\end{array}\right., \quad 
\left\{\begin{array}{l}
\| P_y^s A_{x,y}P_x^u\|_{x,y} \leq  \eta, \\
\| P_y^u A_{x,y}P_x^s\|_{x,y} \leq  \eta,
\end{array}\right. \\ \quad \\
\left\{\begin{array}{l}
\|f_{x,y}(0)\|_y \leq \epsilon(\rho), \\
\forall v \in B_x(\rho), \ \| Df_{x,y}(v) - A_{x,y} \|_{x,y} \leq \eta, 
\end{array}\right. 
\end{gather*}
where $\| \cdot \|_{x,y}$ denotes the matrix norm  computed according to the two adapted local norms $\| \cdot \|_x$ and $\| \cdot \|_y$.
\end{enumerate}
\end{definition}

\begin{definition}[Admissible transitions for maps] \label{definition:AdmissibleTransition}
Let $\Gamma_\Lambda$ be a family of adapted local charts  as given in Definition \ref{Definition:AdaptedLocalCharts}. Let  $x,y \in \Lambda$. We say that $x \overset{\Gamma_\Lambda}{\to} y$ is a $\Gamma_\Lambda$-admissible transition if
\[
f(x) \in \gamma_y(B_y(\epsilon(\rho))) \quad ( \ \Leftrightarrow \ \ f_{x,y}(0) \in B_y(\epsilon(\rho)) \ ).
\]
A sequence $(x_i)_{i=0}^n$ of points of $\Lambda$ is said to be $\Gamma_\Lambda$-admissible if $x_i \overset{\Gamma_\Lambda}{\to}x_{i+1}$ for every $0 \leq i < n$. 
\end{definition}

The existence of a family of adapted local norms is at the heart of the Definition \ref{Definition:AdaptedLocalCharts}. We think it is worthwhile to give a complete proof of the following proposition.

\begin{proposition}
Let $(M,f)$ be a $C^1$ dynamical system and $\Lambda \subseteq M$ be a compact $f$-invariant hyperbolic set. Then there exists a family of adapted local charts $\Gamma_\Lambda=(\Gamma,E,N,F,A)$ together with a set of constants $(\sigma^u,\sigma^s,\eta,\rho)$ as in Definition \ref{Definition:AdaptedLocalCharts}.
\end{proposition}

\begin{proof}
The proof is done into several steps.

\medskip
{Step 1.} We first construct an {\it adapted local norm}. We need the following notion of $(n,R)$-chains.  Let  $n\geq1$ and $R \in (0,1)$. We say that a sequence of points in $\Lambda$, $(x_0,\ldots,x_n)$, is an $(n,R)$-chain,
\[
\forall\, 0 \leq  k < n, \ f(x_k) \in \gamma_{x_{k+1}}(B(R)).
\]
An $(n,0)$-chain is a true orbit, $\forall\, 0 \leq k < n, \ f(x_k)=x_{k+1}$.

Then we choose $\Delta\in(0,1)$ small enough so that,
\[
\forall\, x,y \in \Lambda, \quad \big[ f(x) \in \gamma_y(B(\Delta)) \ \Rightarrow \ f(\gamma_x(B(\Delta)) \subseteq \gamma_y(B(1)) \big].
\]
We choose $N\geq 2$ large enough such that,
\begin{align*}
&\left\{\begin{array}{l}
2 \, C_{\Lambda} \exp(N \lambda^s) \leq \exp(N\kappa^s), \\ 
2 \, C_{\Lambda} \exp(-N \lambda^u) \leq \exp(-N\kappa^u), 
\end{array}\right.
\end{align*}
We choose $ R \in (0,\Delta)$ small enough such that,  for every $(N,R)$-chain $(x_0,\ldots,x_N)$,
\[
\forall\, 0 \leq k  \leq N, \ f^k(\gamma_{x_0}(B(R)) \subseteq \gamma_{x_k}(B(\Delta)).
\]

We equipped $\mathbb{R}^d$ with the pull backward by $T_0\gamma_x$ of the initial Finsler norm on each $T_xM$ that we call $\| \cdot \|_x^{\star}$. Thanks to the equivariance and the continuity of $E_\Lambda^u(x) \oplus E_\Lambda^s(x)$, we may choose $R$ sufficiently small such that,
\[
\left\{\begin{array}{l}
\| P^s_{x_N}A_{x_{N-1},x_N}P^s_{x_{N-1}} \cdots P^s_{x_1}A_{x_0,x_1}P^s_{x_0} \|_{x_0,x_N}^\star \leq \exp(N\kappa^s), \\
\| ( P^u_{x_N}A_{x_{N-1},x_N}P^u_{x_{N-1}} \cdots P^u_{x_1}A_{x_0,x_1}P^u_{x_0})^{-1} \|_{x_0,x_N}^\star \leq \exp(-N\kappa^u).
\end{array}\right.
\]
The adapted local norm $\|\cdot \|_x$  is by definition the norm on $E_s^u \oplus E_x^s$  defined by,
\begin{enumerate}
\item $\forall\,v \in E_x^u, \ w \in E_x^s, \quad \|v+w \|_x = \max(\|v\|_x,\|w\|_x)$,
\item $\|v\|_x :=  \max_{1 \leq k < N}  \sup_{(x_k,\ldots,x_N), \, \text{$(N-k,R)$-chain}, \, x_N=x}  \\ 
\phantom{\|v\|_x :=}\Big( \|v\|,\|  (P^u_{x_{N}}A_{x_{N-1},x_{N}}P^u_{x_{N-1}} \cdots P^u_{x_{k+1}}A_{x_k,x_{k+1}} P^u_{x_k})^{-1}v \|_{x_k,x_N}^\star e^{(N-k)\kappa^u} \Big)$,
\item $\|w\|_x :=  \max_{1 \leq k < N}  \sup_{(x_0, \ldots,x_k), \, \text{$(k,R)$-chain}, \, x_0=x}  \\ 
\phantom{\|w\|_x :=}\Big( \|w\|,\|  P^s_{x_{k}}A_{x_{k-1},x_{k}}P^s_{x_{k-1}} \cdots P^s_{x_{1}}A_{x_0,x_{1}} P^s_{x_0} w \|_{x_0,x_k}^* e^{-k\kappa^s} \Big)$,
\end{enumerate}
where the supremum is taken over all $(N-k,R)$-chains $(x_k,\ldots,x_N)$ ending at $x$ for the unstable norm, and $(k,R)$-chains $(x_0,\ldots,x_k)$ starting from $x$ for the stable norm, of any length $1 \leq k < N$. Let $B_x(\epsilon)$ be the ball of radius $\epsilon$ for the norm $\| \cdot \|_x$. We finally choose $\rho < R$ small enough so that  for every $x \in \Lambda$,
\begin{gather*}
B_x(\rho) \subseteq B(R), 
\end{gather*}
and  for every $x,y \in \Lambda$ satisfying $f_{x,y}(0) \in B_y(\rho)$,
\begin{gather*}
\forall v \in B_x(\rho), \ \| Df_{x,y}(v) - A_{x,y} \|_{x,y} < \eta.
\end{gather*}
Thanks to the equivariance of the unstable and stable vector bundles, we choose $\rho$ small enough so that
\[
\| P_y^s A_{x,y}P_x^u\|_{x,y} <  \eta \quad \text{and} \quad
\| P_y^u A_{x,y}P_x^s\|_{x,y} <  \eta.
\]

\medskip
{\it Step 2.} We  prove the inequalities,
\begin{gather*}
\forall\, v \in E_x^u, \ \|A_{x,y}v\|_y \geq \sigma^u \|v\|_x
\quad \text{and} \quad
\forall\, v \in E_x^s, \ \|A_{x,y}v \|_y \leq  \sigma^s \|v\|_x.
\end{gather*}
We prove the second inequality with $\sigma^s$, the other inequality with $\sigma^u$ is similar. Let  $v \in E^s_x$ of norm $\|v\|_x=1$ and $w = P^s_y A_{x,y}v$. We discuss 3 cases.

Either $\|w\|_y = \|w\|$, $(x,y)$ is an $(1,R)$-chain, then
\[
\|w\|_y = \|P_y^s A_{x,y}P^s_x v\|  = \big(\|P_y^s A_{x,y}P^s_x v\| e^{-\kappa^s}\big) e^{\kappa^s} \leq \|v\|_x e^{\kappa^s}.
\]

Or there exists $1 \leq k<N-1$ and an $(k,R)$-chain $(y_0, \ldots,y_k)$  such that  $y=y_0$ and
\[
\|w\|_y = \|  P^s_{y_{k}}A_{y_{k-1},y_{k}}P^s_{y_{k-1}} \cdots P^s_{y_{1}}A_{y_0,y_{1}} P^s_{y_0} w \| e^{-k\kappa^s}.
\]
Then $(x,y,y_1,\ldots,y_{k})$ is an $(k+1,R)$-chain of length $k+1 < N$,
\[
\|w\|_y =  \|  P^s_{y_{k}}A_{y_{k-1},y_{k}}P^s_{y_{k-1}} \cdots P^s_{y_{1}}A_{y_0,y_{1}} P^s_{y_0} A_{x,y_0}P^s_x \| e^{-(k+1)\kappa^s} e^{\tilde\lambda^s} \leq \|v\|_x e^{\tilde\lambda^s}.
\]

Or there exists an $(N-1,R)$-chain   $(y_0, \ldots,y_{N-1})$  such that $y_0=y=$ and
\[
\|w\|_y = \|  P^s_{y_{N-1}}A_{y_{N-2},y_{N-1}}P^s_{y_{N-2}} \cdots P^s_{y_{1}}A_{y_0,y_{1}} P^s_{y_0} w \| e^{-(N-1)\kappa^s}.
\]
Then $(x,y_0,\ldots,y_{N-1})$ is an $(N,R)$-chain, and by the choice of $N$ 
\[
\|  P^s_{x_{N-1}}A_{x_{N-2},x_{N-1}}P^s_{x_{N-2}} \cdots P^s_{x_{1}}A_{x_0,x_{1}} P^s_{x_0} A_{x,x_0} \| \leq e^{N\kappa^s}.
\]
We thus obtain
\begin{align*}
\|w\|_y &\leq \|  P^s_{x_{N-1}}A_{x_{N-2},x_{N-1}}P^s_{x_{N-2}} \cdots P^s_{x_{1}}A_{x_0,x_{1}} P^s_{x_0} A_{x,x_0} \| e^{-(N-1)\kappa^s} \, \|v\| \\
&\leq \|v\| \, e^{\kappa^s} \leq \|v\|_x \, e^{\kappa^s}.
\end{align*}
In the 3 cases we have proved $\| P^s_y A_{x,y}v \|_y \leq \|v\|_x e^{\kappa^s}$ or $\|P^s_y A_{x,y}P^s_x \|_{x,y} \leq \sigma^s$.
\end{proof}

\subsection{Adapted local unstable manifold}

We review in this section the property of stability of cones under the iteration of a hyperbolic map. We recall the forward stability of unstable cones, and the backward stability of stable cones.

\begin{definition}[Unstable/stable cones]
Let  $\mathbb{R}^d = E^u \oplus E^s$ be a splitting equipped with a Banach norm $| \cdot |$. Let  $\alpha \in  (0,1)$ 
\begin{enumerate}
\item The {\it unstable cone of angle $\alpha$} is the set
\[
\mathcal{C}^u(\alpha) := \big\{ w \in \mathbb{R}^d : | P^s w | \leq \alpha |P^uw| \big\}.
\]
\item The {\it stable cone of angle $\alpha$} is the set
\[
\mathcal{C}^s(\alpha) := \big\{ w \in \mathbb{R}^d : | P^u w | \leq \alpha |P^s w| \big\}.
\]
\end{enumerate}
\end{definition}

Notice that the unstable cone $\mathcal{C}^u(\alpha)$ contains the unstable vector space $E^u$ and similarly for the stable cone.

\begin{lemma}[Equivariance of unstable cones] \label{Lemma:EquivarianceUnstableCone}
We consider the notations of Definition \ref{Definition:AdaptedLocalHyperbolicMap}, where $(\sigma^u,\sigma^s,\rho,\eta)$ are the set of hyperbolic constants, $\mathbb{R}^d = E^u \oplus E^s$ and $\mathbb{R}^d = \tilde E^u \oplus \tilde E^s$ are two Banach spaces with norms $| \cdot |$ and $\| \cdot \|$ respectively, and $(A,f,E^{u/s}, \tilde E^{u/s}, | \cdot |,\| \cdot\|)$ is an adapted local hyperbolic map. Let 

\[
\alpha  \in \Big(\frac{6\eta}{\sigma^u-\sigma^s},1\Big) \quad\text{and}\quad \beta := \frac{\alpha \sigma^s+3\eta}{\sigma^u-3\eta}.
\]
Then $\beta \leq \alpha$ and, for every $a,b \in B(\rho) =  B^u(\rho) + B^s(\rho)$,
\begin{enumerate}
\item if $b-a \in \mathcal{C}^u(\alpha)$, then
\[
f(b)-f(a) \in \tilde{\mathcal{C}}^u(\beta) \quad\text{and}\quad \|\tilde P^u(f(b)-f(a)) \| \geq (\sigma^u-3\eta) |P^u(b-a)|,
\]
\item if $f(b) -f(a) \in \tilde{\mathcal{C}}^s(\alpha)$, then
\[
b-a \in \mathcal{C}^s(\beta)  \quad\text{and}\quad  \| \tilde P^s(f(b)-f(a)) \| \leq (\sigma^s+3\eta)  |P^s(b-a)|.
\]
\end{enumerate}
\end{lemma}

We recall the existence of local unstable manifolds. We are not assuming $f$  invertible. In particular the local stable manifold may not exist. We choose a sequence of admissible transitions and prove the equivalence between two definitions.

\begin{definition} \label{Definition:LocalUnstableManifold}
Let $\Gamma_\Lambda$ be a family of adapted local charts. Let $\underline{x} = (x_i)_{i \in \mathbb{Z}}$ be a sequence of $\Gamma_\Lambda$-admissible transitions, $\forall\, i \in \mathbb{Z}, \ x_i \overset{\Gamma_\Lambda}{\to} x_{i+1}$. Denote $f_i := f_{x_i,x_{i+1}}$, $E_i^{u/s} = E_{x_i}^{u/s}$ and $\| \cdot \|_{i} = \| \cdot \|_{x_i}$. Then $(f_i,A_i,E_i^{u/s},\| \cdot \|_i)$ is an adapted local hyperbolic map. The {\it local unstable manifold at the position $i$} is the set
\[
W_i^u(\underline{x}) = \big\{ q \in B_i(\rho) : \exists (q_k)_{k \leq i}, \ q_i=q, \ \forall\, k<i, \ q_k \in B_k(\rho), \ \text{and} \ f_k(q_k) = q_{k+1} \big\},
\]
where $B_i(\rho) = B_i^u(\rho) \oplus B_i^s(\rho)$ is the ball with respect to the adapted local  norm $\| \cdot \|_i$.
\end{definition}

The following theorem shows that, observed in adapted local charts, the local unstable manifolds have a definite size and the local maps expand uniformly.

\begin{theorem}[Adapted local unstable manifold] \label{Theorem:AdaptedLocalUnstableManifold}
Let $\Gamma_\Lambda$ be a family of adapted local charts, and  $\underline{x} = (x_i)_{i \in \mathbb{Z}}$ be a sequence of $\Gamma_\Lambda$-admissible transitions. Let  $f_i = f_{x_i,x_{i+1}}$ be the local maps, $\| \cdot \|_i$ be the local norms, and $\mathcal{G}^u_i$ be the set of Lipschitz graphs as in Definition \ref{Definition:LocalLipschitGraph},
\[
\mathcal{G}_i^u := \Big\{ [G : B_i^u(\rho) \to B_i^s(\rho)] : \Lip(G) \leq  \frac{6\eta}{\sigma^u - \sigma^s}, \ \|G(0)\|_i \leq  \frac{\rho}{2} \Big\}.
\]
Let $0_i^u$ be the null graph in the ball $B_i(\rho)$, and
\[
G_i^n := (\mathcal{T})^u_{i-1} \circ \cdots \circ (\mathcal{T})^u_{i-n+1} \circ (\mathcal{T})^u_{i-n}(0_{i-n}^u).
\]
Then
\begin{enumerate}
\item $(G_i^n)_{n\geq1}$ converges uniformly to a Lipschitz graph $[G_i^u : B_i^u(\rho) \to B_i^s(\rho)]$.
\item The local unstable manifold defined in \ref{Definition:LocalUnstableManifold} coincides with $\Graph(G_i^u)$:
\[
W_i^u(\underline{x}) = \Graph(G_i^u) = \{ v + G_i^u(v) : v \in B_i^u(\rho) \}.
\]
\item The local unstable manifold is equivariant in the sense:
\[
\forall\, i\in\mathbb{Z}, \quad f_i(\Graph(G_i^u)) \supseteq \Graph(G_{i+1}^u),
\]
or more precisely $(\mathcal{T})^u_i(G_i^u) = G_{i+1}^u$.
\item The local unstable manifold is Lipschitz:
\[
\Lip(G_i^u) \leq \frac{6\eta}{\sigma^u-\sigma^s}.
\]
\item The adapted maps are uniformly expanding:
\[
\forall\, i \in\mathbb{Z}, \ \forall\, q,q' \in \Graph(G^u_i), \quad \| f_i(q) - f_i(q') \|_{i+1} \geq (\sigma^u-3\eta) \| q-q'\|_i.
\]
\end{enumerate}
\end{theorem}


\end{document}